\documentclass[a4paper, reqno]{amsart} 
\usepackage[dvipdfmx]{graphicx} 
\usepackage{times,txfonts,amsmath,amssymb} 
 
\theoremstyle{plain} 
\newtheorem{thm}{Theorem}[section] 
\newtheorem{lem}[thm]{Lemma} 
\newtheorem{cor}[thm]{Corollary} 
\newtheorem{prop}[thm]{Proposition} 
\newtheorem{conj}[thm]{Conjecture} 
 
\theoremstyle{definition} 
\newtheorem{defn}[thm]{Definition} 
\newtheorem{rem}[thm]{Remark}

\newcommand{\cA}{{\mathcal A}} 
\newcommand{\cB}{{\mathcal B}} 
\newcommand{\cC}{{\mathcal C}} 
\newcommand{\cF}{{\mathcal F}} 
\newcommand{\cI}{{\mathcal I}} 
\newcommand{\cR}{{\mathcal R}} 
\newcommand{\cS}{{\mathcal S}} 
 
\newcommand{\bP}{{\mathbb P}} 
\newcommand{\bZ}{{\mathbb Z}} 
\newcommand{\bC}{{\mathbb C}} 
\newcommand{\bQ}{{\mathbb Q}} 
\newcommand{\bR}{{\mathbb R}} 
\newcommand{\on}[1]{{\operatorname{#1}}} 
 
\title{ 
The smallest line arrangement which is free but 
not recursively free} 
\author[T.~Abe]{Takuro Abe} 
\address{ 
Takuro Abe 
\newline 
Department of Mechanical Engineering and Science, Kyoto University 
\newline 
Yoshida-Honmachi, Sakyo-ku, 
Kyoto 606-8501, JAPAN 
} 
\email{abe.takuro.4c@kyoto-u.ac.jp} 
\author[H.~Kawanoue]{Hiraku Kawanoue} 
\address{Hiraku Kawanoue 
\newline 
Research Institute for Mathematical Sciences, 
Kyoto University 
\newline 
Kitashirakawa-Oiwakecho, Sakyo-ku, 
Kyoto 606-8502, JAPAN 
} 
\email{kawanoue@kurims.kyoto-u.ac.jp} 
\author[T.~Nozawa]{Takeshi Nozawa} 
\address{ 
Takeshi Nozawa 
\newline 
Maizuru National College of Technology 
\newline 
234 Shiraya, Maizuru 625-8511, JAPAN 
} 
\email{nozawa@kurims.kyoto-u.ac.jp} 
\thanks{ 
The first author's work was partially supported by Japan Society for the 
Promotion of Science 
Grant-in-Aid for Young Scientists (B), No.~24740012. 
\newline\indent 
The second author's work was partially supported by Japan Society for the 
Promotion of Science 
Grant-in-Aid for Young Scientists (B), No.~23740016. 
} 
\begin{document} 
\begin{abstract} 
In the category of free arrangements, inductively and recursively free 
arrangements are 
important. In particular, in the former, the conjecture by Terao 
asserting that freeness 
depends only on combinatorics holds true. A long standing problem 
whether all free 
arrangements are recursively free or not is settled by Cuntz and Hoge 
very recently, by 
giving a free but non-recursively free plane arrangement consisting of 
27 planes. 
 
In this paper, we construct a free but non-recursively free plane 
arrangement consisting 
of 13 planes, and show that this example is the smallest in the sense of 
the cardinality of 
planes. In other words, all free plane arrangements consisting of at 
most 12 planes are 
recursively free. To show it, we completely classify all free plane 
arrangements in terms of 
inductive freeness and three exceptions when the number of planes is at 
most 12. 
\end{abstract} 
\maketitle 
\begin{section}{Introduction} 
In the study of hyperplane arrangements, 
one of the most important problems is 
the freeness of them. In general, to determine whether 
a given arrangement is free or not is very difficult, 
and there is essentially only one way to check it, Saito's criterion (Theorem \ref{saito}). 
On the other hand, there is a nice way to construct a free arrangement 
from a given free arrangement, called 
the addition-deletion theorem (Theorem \ref{ad}). Since the empty arrangement 
is free, 
there is a natural question whether every free arrangement can be 
obtained, starting from empty arrangement, 
by applying addition and deletion theorems. For simplicity, 
for the rest of this section, let us 
concentrate our interest on the central arrangements in $\bC^3$. 
 
We say that a central arrangement $\cA$ is {\it inductively free} 
if it can be constructed by using only the addition theorem from 
the empty arrangement, and {\it recursively free} if we use 
both the addition and deletion theorems to construct it. It was very 
soon to be found a free arrangement which is not inductively free (see 
Example 4.59 in \cite{OT} for example). 
However, a free but non-recursively free arrangement has not been 
found for a long time. 
It was very recent that Cuntz and Hoge first found such an example in \cite{CH}, 
which consists of 27 planes over $\bQ(\zeta)$ with the fifth root 
of unity $\zeta$ in $\bC$. 
 
The aim of this paper is to give a new example of free 
but non-recursively free arrangements in $\bC^3$ consisting of 
the smallest number of planes. Our example consists of 
13 planes defined over $\bQ[\sqrt{3}]$.\footnote{ 
After posting the first version of this paper, 
we were informed by Professor Cuntz that 
he also found this example independently.  See \cite{C}.} 
To show that there are 
no such arrangements when the number of planes is 
strictly less than 13, we also investigate the set of all 
free arrangements $\cA$ with $|\cA| \le 12$. 
In other words, we give the complete classification of such 
free arrangements in terms of inductive, recursive freeness 
and three exceptions 
given in Definitions \ref{dH}, \ref{Pen} and \ref{443}. Now let us state 
our main theorem in the following. 
 
\begin{thm}\label{Main} 
Let $\cA$ be a central arrangement in $\bC^3$. 
\item[(1)] 
If $\cA$ is free with $|\cA|\leq12$, 
then $\cA$ is recursively free. More precisely, 
$\cA$ is either inductively free 
or one of the following arrangements 
characterized by their lattice structures. 
\begin{enumerate}\renewcommand{\labelenumi}{(\roman{enumi})} 
\item A dual Hesse arrangement, i.e., the arrangement $\cA$ 
with $|\cA|=9$, 
whose set $L_2(\cA)$ of codimension $2$ intersections 
consists of $12$ triple lines. 
\item A pentagonal arrangement, i.e., the arrangement $\cA$ 
with $|\cA|=11$ such that 
$L_2(\cA)$ consists of 
$10$ double lines, $5$ triple lines, 
$5$ quadruple lines and that 
any $H\in\cA$ contains at most 
$5$ lines of $L_2(\cA)$. 
\item 
A monomial arrangement associated to the group $G(4,4,3)$, 
i.e., the arrangement $\cA$ 
with $|\cA|=12$, 
whose $L_2(\cA)$ consists of $16$ triple lines and $3$ 
quadruple lines. 
\end{enumerate} 
Moreover, the lattices of arrangements in (i), (ii) and 
(iii) are realized over 
$\bQ[\sqrt{-3}]$, $\bQ[\sqrt{5}]$ and $\bQ[\sqrt{-1}]$, 
respectively. In particular, all free arrangements 
in $\bC^3$ are inductively free 
when $|\cA|\le 8$ or $|\cA|=10$. 
\item[(2)] 
There exists a free but not recursively free 
arrangement $\cA$ 
over $\bQ[\sqrt{3}]$ with $|\cA|=13$. 
\end{thm} 
Our proof is based on the combinatorial method. 
We can check Theorem \ref{Main} by easy computations by hand, 
or just drawing a nice picture of our arrangement. 
 
Theorem \ref{Main} contains a characterization of 
free arrangements $\cA$ in $\bC^3$ 
with $|\cA|\le 12$. There have been several researches 
on this way, especially from the viewpoint of 
the conjecture by Terao (Conjecture \ref{TC}), 
which asks whether 
the freeness of an arrangement depends only on its combinatorics. 
For example, 
the conjecture by Terao was checked when $|\cA| \le 11$ in \cite{WY}, 
and $|\cA|\le 12 $ in \cite{FV}. 
By Theorem \ref{Main}, we can give another proof 
of the conjecture by Terao 
when $|\cA|\le 12$ which is originally due to Faenzi and Vall\`{e}s. 
\begin{cor}[Faenzi-Vall\`es \cite{FV}] 
The conjecture by Terao holds for central arrangements 
$\cA$ in $\bC^3$ with $|\cA|\leq12$. 
\end{cor} 
By investigating the structure of the classification in 
Theorem \ref{Main}, we can say that almost all the free 
arrangements with small exponents are inductively free. 
\begin{cor}\label{le4} 
Let $\cA$ be a free arrangement in $\bC^3$ 
with $\exp(\cA)=(1,a,b)$. 
If $\min(a,b)\leq4$, then 
$\cA$ is either inductively free or 
a dual Hesse arrangement appearing in Theorem \ref{Main}. 
\end{cor} 
 
The organization of our paper is as follows. 
In \S \ref{prel} we introduce several definitions and results 
which will be used in this paper. 
In \S \ref{S9}, \S \ref{S11} and \S \ref{S12} we prove Theorem \ref{Main} (1). 
In \S \ref{13} we prove Theorem \ref{Main} (2). 
 
\subsection*{Acknowledgements} 
We are grateful to M.~Yoshinaga for his helpful comments. 
\end{section} 
\begin{section}{Preliminaries}\label{prel} 
In this section, we summarize several definitions and results 
which will be used in this 
paper. 
For the basic reference on the arrangement theory, 
we refer Orlik-Terao \cite{OT}.

Let $V=\bC^n$. An {\it arrangement of hyperplanes} $\cA$ is 
a finite set of affine hyperplanes in $V$. An arrangement $\cA$ 
is {\it central} if every hyperplane is linear. 
For a hyperplane $H \subset V$, define 
$$ 
\cA \cap H=\{H \cap H' \neq \emptyset \mid H' \in \cA,\ H ' \neq H\}. 
$$ 
Hence $\cA \cap H$ is an arrangement in an $(n-1)$-dimensional vector space $H$. 
Let us define a {\it cone} $\on{c}\cA$ of an affine arrangement 
$\cA$ as follows. If $\cA$ is defined by a 
polynomial equation $Q=0$, then $\on{c}\cA$ is defined by 
$z\cdot\on{c}Q=0$, where $\on{c}Q$ is the homogenized 
polynomial of $Q$ by the new coordinate $z$. 
When $\cA$ is central, let us fix a defining linear form 
$\alpha_H \in V^*$ for each $H \in \cA$. 
 
From now on, let us concentrate our interest on 
central arrangements in $\bC^n$ when $n=2$ or $3$. 
So arrangements of lines or planes. 
Even when $n=3$, identifying $\bC^3$ with $P_{\bC}^2$, we also say 
they are line arrangements when there are 
no confusions. Let $S=S(V^*)=\bC[x_1,\ldots,x_n]$ be 
the coordinate ring of $V$. 
For a central arrangement $\cA$, define 
$$ 
D(\cA)=\left\{ \theta \in \bigoplus_{i=1}^n S \partial/\partial {x_i} \mid 
\theta(\alpha_H) \in S \alpha_H\ \text{for all}\ H \in \cA\right\}. 
$$ 
$D(\cA)$ is called the {\it logarithmic derivation module}. 
$D(\cA)$ is reflexive, but not free in general. 
We say that $\cA$ is free with {\it exponents} $(d_1,\ldots,d_n)$ if 
$D(\cA)$ has a homogeneous free basis $\theta_1,\ldots,\theta_n$ with 
$\deg \theta_i=d_i\ (i=1,\ldots,n)$. 
Here the {\it degree} of a homogeneous derivation 
$\theta=\sum_{i=1}^n f_i \partial/\partial {x_i}$ 
is defined by $\deg f_i$ for all non-zero $f_i$. 
Note that the Euler derivation $\theta_E= 
\sum_{i=1}^n x_i \partial/\partial x_i$ is contained in 
$D(\cA)$. 
In particular, it is easy to show that 
$ 
D(\cA)$ has $S \theta_E $ as its direct summand 
for a non-empty arrangement $\cA$. 
Hence $\exp(\cA)$ always contains $1$ if $\cA$ is not empty. 
 
To verify the freeness of $\cA$, the following 
Saito's criterion is essential. 
 
\begin{thm}[Saito's criterion, \cite{S}] 
Let $\theta_1,\ldots,\theta_n \in D(\cA)$. Then the following 
three conditions are equivalent. 
\begin{itemize} 
\item[(1)] 
$\cA$ is free with basis $\theta_1,\ldots,\theta_\ell$. 
\item[(2)] $\det [\theta_i(x_j)]=c \prod_{H \in \cA} \alpha_H$ 
for some non-zero $c \in \bC$. 
\item[(3)] $\theta_1,\ldots,\theta_n$ are all homogeneous derivations, 
$S$-independent and 
$\sum_{i=1}^n \deg\theta_i=|\cA|$. 
\end{itemize} 
\label{saito} 
\end{thm} 
 
For a {\it multiplicity} $m\colon\cA \rightarrow \bZ_{>0}$, 
we can define the logarithmic 
derivation module $D(\cA,m)$ of a {\it multiarrangement} $(\cA,m)$ by 
$$ 
D(\cA,m)=\left\{ 
\theta \in \bigoplus_{i=1}^n S \partial/\partial {x_i} \mid 
\theta(\alpha_H) \in S \alpha_H^{m(H)}\ \text{for all}\ H \in \cA 
\right\}. 
$$ 
The freeness, exponents, and Saito's criterion 
for a multiarrangement can be defined in the same manner as 
an arrangement case. Note that the Euler derivation is not 
contained in $D(\cA,m)$ in general. 
 
Recall that every central (multi)arrangement in $\bC^2$ is free 
since $\dim_\bC \bC^2=2$ and $D(\cA ,m)$ is reflexive. Hence the first 
non-free central arrangement occurs when $n=3$. 
Define the {\it intersection lattice} $L(\cA)$ of $\cA$ by 
$$ 
L(\cA)=\left\{ \bigcap_{H \in \cB} H \mid \cB \subset \cA\right\}. 
$$ 
$L(\cA)$ has a poset structure with an order 
by the reverse inclusion. $L(\cA)$ is considered 
to be a combinatorial information of $\cA$. 
 
Define the {\it M\"{o}bius function} 
$\mu\colon L(\cA) \rightarrow \bZ$ by 
$$ 
\mu(V)=1,\quad 
\mu(X)=-\sum_{Y \in L(\cA),\ X \subsetneqq Y \subset V} \mu(Y) 
\quad (X \neq V). 
$$ 
Then the {\it characteristic polynomial} $\chi(\cA,t)$ 
of $\cA$ is defined by 
$\chi(\cA,t)=\sum_{X \in L(\cA)} \mu(X)\ t^{\dim X}$. 
It is known that 
$$ 
\chi(\cA,t)=t^n\on{Poin}(\bC^n \setminus \cup_{H \in \cA} H,-t^{-1}). 
$$ 
Hence $\chi(\cA,t)$ is both combinatorial and topological invariants of an arrangement. 
The freeness and $\chi(\cA,t)$ are related by the following formula. 
 
\begin{thm}[Factorization, \cite{T2}]\label{factorization} 
Assume that a central arrangement $\cA$ is free with 
$\exp(\cA)=(d_1,\ldots,d_n)$. Then 
$$ 
\chi(\cA,t)=\prod_{i=1}^n (t-d_i). 
$$ 
\end{thm} 
Theorem \ref{factorization} implies that the algebra of 
an arrangement might control the combinatorics and the topology of it. 
However, the converse is not true in general. For example, there 
is a non-free arrangement in $\bC^3$ 
whose characteristic polynomial factorizes over 
the ring of integers (see \cite{OT}). Hence it is natural to ask 
how much algebra and combinatorics of arrangements are related. 
 
The following conjecture is one of the largest ones in 
the theory of arrangements. 
 
\begin{conj}[Terao]\label{TC} 
The freeness of an arrangement depends only on the combinatorics. 
\end{conj} 
 
Conjecture \ref{TC} is open even when $n=3$. In \cite{T}, 
Terao introduced a nice family of 
free arrangements in which Conjecture \ref{TC} holds. 
To state it, let us introduce the following 
key theorem in this paper. 
 
\begin{thm}[Addition-Deletion, \cite{T}] \label{ad} 
Let $\cA$ be a free arrangement in $\bC^3$ with $\exp(\cA)=(1,d_1,d_2)$. 
\item[(1)](the addition theorem). Let $H \not \in \cA$ be a linear plane. 
Then $\cB=\cA \cup\{H\}$ is free with $\exp(1,d_1,d_2+1)$ if and only if 
$|\cA \cap H|=1+d_1$. 
\item[(2)](the deletion theorem). Let $H \in \cA$. 
Then $\cA'=\cA \setminus \{H\}$ is free with $\exp(1,d_1,d_2-1)$ if and only if 
$|\cA' \cap H|=1+d_1$. 
\end{thm} 
\begin{defn}\label{IR} 
\item[(1)] A central plane arrangement $\cA$ is {\it inductively free} 
if there is a filtration of subarrangements 
$\cA_1 \subset \cA_2 \subset \cdots \subset \cA_\ell=\cA$ such that 
$|\cA_i|=i\ (1\leq i\leq\ell)$ and every $\cA_i$ is free. 
\item[(2)] 
A central plane arrangement $\cA$ is {\it recursively free} 
if there is a sequence of arrangements 
$\emptyset=\cB_0,\cB_1,\cB_2,\ldots,\cB_t=\cA$ 
such that $||\cB_{i+1}|-|\cB_i||=1\ (1\leq i\leq t-1)$ and 
every $\cB_i$ is free. 
\end{defn} 
 
Roughly speaking, an inductively free arrangement is 
a free arrangement constructed from an empty arrangement by 
using only the addition theorem, and 
a recursively free arrangement is a free arrangement constructed 
from an empty arrangement by using both 
the addition and deletion theorems. 
It is known that in the category of inductively free 
arrangements, Conjecture \ref{TC} is true, but open in that of 
recursively free arrangements. 
\begin{rem} 
Theorem \ref{ad} and Definition \ref{IR} are different from 
those in an arbitrary dimensional case. They coincide when $n=3$ 
because every arrangement in $\bC^2$ is free. 
For a general definition, see \cite{T} and \cite{OT} for example. 
\end{rem} 
\begin{defn} 
For $\ell\in\bZ_{\geq0}$, we define the sets $\cF_\ell$, 
$\cI_\ell$ and $\cR_\ell$ as follows. 
\begin{eqnarray*} 
\cF_{\ell}&=&\left\{ 
\text{free arrangement}\ \cA\ \text{in}\ \bC^3 
\ \text{with }\ |\cA|=\ell\right\}, 
\\ 
\cI_{\ell}&=&\left\{\cA\in\cF_\ell\mid 
\cA\colon\text{inductively free} 
\right\}. 
\\ 
\cR_{\ell}&=&\left\{\cA\in\cF_\ell\mid 
\cA\colon\text{recursively free} 
\right\}. 
\end{eqnarray*} 
Note that $\cI_{\ell}\subset\cR_{\ell}\subset\cF_{\ell}$ 
by definition. 
\end{defn} 
 
For the rest of this section, we concentrate our interest on 
a central arrangements in $\bC^3$. 
One of the main purposes of this paper is to clarify the difference 
between $\cI_{\ell}$ and $\cF_{\ell}$ for $0\leq\ell\leq12$. 
Since 
$\cI_{\ell}=\cF_{\ell}$ for $\ell\leq1$, 
we may assume $\ell\geq2$. 
For $H\in\cA$, we denote 
$ 
n_{\cA,H}=\left|\cA\cap H\right| 
$. 
The following is 
the foundation stone of our analysis in this paper. 
\begin{thm}[\cite{A}]\label{ABT} 
Assume $\chi(\cA,t)=(t-1)(t-a)(t-b)$ for 
$a,b\in\bR$ 
and there exists $H\in\cA$ such that $n_{\cA,H}>\min(a,b)$. 
Then, $\cA$ is free if and only if $n_{\cA,H}\in\{a+1,b+1\}$. 
\end{thm} 
\begin{defn} 
In view of theorem \ref{ABT}, we define 
the subset $\cS_{\ell}$ of $\cF_{\ell}$ as 
$$\cS_{\ell}=\{\cA\in\cF_\ell 
\mid \exp(\cA)=(1,a,b),\ \max_{H\in\cA}n_{\cA,H}\leq\min(a,b)\}. 
$$ 
Note that $\cS_{\ell}\cap\cI_{\ell}=\emptyset$ by 
Theorem \ref{ad}. 
\end{defn} 
\begin{lem}\label{Red} 
If $\cF_{\ell-1}=\cI_{\ell-1}$ 
(resp.~$\cR_{\ell-1}$), 
then 
$\cF_{\ell}=\cS_{\ell}\sqcup \cI_{\ell}$ 
(resp.~$\cS_{\ell}\cup \cR_{\ell}$). 
\end{lem} 
\begin{proof} 
Take $\cA\in\cF_{\ell}\setminus\cS_{\ell}$ 
and set $\exp(\cA)=(1,a,b)$. 
By Theorem \ref{ABT}, 
$\cA\in\cF_{\ell}\setminus\cS_{\ell}$ 
if and only if there exists 
a line $H\in\cA$ such that $n_{\cA,H}=a+1$ or $b+1$. 
Therefore, by the deletion theorem, 
we have $\cA\setminus\{H\}\in \cF_{\ell-1}$. 
Now the assertions above are clear. 
\end{proof} 
 
In the rest of this paper, 
we regard a central arrangement in $\bC^3$ 
as a line arrangement in $\bP_{\bC}^2$. 
For a line arrangement $\cA$ in $\bP_{\bC}^2$ 
and $P\in\bP_{\bC}^2$, we set 
$$ 
\cA_P=\{H\in\cA\mid P\in H\},\ 
\mu_P(\cA)=\left|\cA_P\right|-1,\ 
\mu_{\cA}=\sum_{P\in\bP_{\bC}^2}\mu_P(\cA). 
$$ 
Note that $\mu_P(\cA)$ is the reformulation of 
M\"obius function for $L_2(\cA)$. 
If $\cA\neq\emptyset$, we can express $\chi({\cA},t)$ 
as follows by definition. 
$$\chi({\cA},t)=(t-1)\left\{ 
t^2-(|\cA|-1)(t+1) 
+\mu_{\cA}\right\}.$$ 
Concerning the set $\cS_{\ell}$, we have the following lemma. 
\begin{lem}\label{a-2} 
Let $\cA\in\cF_{\ell}$ with $\ell\geq2$ and 
$\exp(\cA)=(1,a,b)$. 
Assume $\mu_P(\cA)\geq\min(a,b)-1$ for some $P\in\bP_{\bC}^2$. 
Then we have $\cA\not\in\cS_{\ell}$.  In particular, 
$\cS_{\ell}=\emptyset$ for $2\leq\ell\leq6$. 
\end{lem} 
\begin{proof} 
We assume $a\leq b$ and $\cA\in\cS_{\ell}$. 
Let $P_0$ be a point in $\bP_{\bC}^2$. 
 
Suppose $\mu_{P_0}(\cA)\geq a$. 
If there exists $H\in \cA\setminus\cA_{P_0}$, we have 
$n_{\cA,H}\geq|\cA_{P_0}\cap H|\geq a+1$, which contradicts 
to $\cA\in\cS_{\ell}$. Therefore $\cA_{P_0}=\cA$, and it 
follows that $\exp(\cA)=(1,0,\ell-1)$. 
Thus $a=0$, but there exists $H\in\cA$ with $n_{\cA,H}\geq1$ 
since $\ell\geq2$, which contradicts to $\cA\in\cS_{\ell}$. 
 
Suppose $\mu_{P_0}(A)=a-1$. 
Since $|\cA_{P_0}|=a$ and $\cA\in\cS_{\ell}$, 
all intersection points of $\cA$ 
lie on $\bigcup_{H\in\cA_{P_0}}H$. 
It follows that $\cA$ is super solvable 
(see Definition 2.32 of \cite{OT} for the details) 
and $\exp(\cA)=(1,a-1,\ell-a)$, 
which contradicts to the condition on $a$. 
 
If $\ell\leq6$, the condition for the lemma automatically 
holds since we have 
$\mu_P(\cA)\geq1$ for some $P\in\bP_{\bC}^2$ and 
$a\leq\lfloor(\ell-1)/2\rfloor\leq2$. 
Therefore we have $\cS_{\ell}=\emptyset$. 
\end{proof} 
 
Now we introduce the invariant $F(\cA)$, which will be used 
to classify $\cS_\ell$. 
\begin{defn} 
Let $\cA$ be a line arrangement in $\bP_{\bC}^2$. We denote 
$$M_i(\cA)=\{P\in{\bP_{\bC}}^2\mid \mu_P(\cA)=i\}$$ 
and set the invariant $F(\cA)$ as 
$$ 
F(\cA)=[F_1(\cA),F_2(\cA),\ldots],\quad 
F_i(\cA)=\left|M_i(\cA)\right|\ (i=1,2,\ldots). 
$$ 
\end{defn} 
\begin{lem} 
The invariant $F(\cA)$ satisfies the following formulae. 
$$ 
\sum_{i} 
iF_i(\cA)=\mu_\cA 
,\ 
\sum_{i} 
(i+1)F_i(\cA)=\sum_{H\in\cA}n_{\cA,H} 
,\ 
\sum_{i} 
\binom{i+1}{2}F_i(\cA)=\binom{|\cA|}{2} 
. 
$$ 
\end{lem} 
\begin{proof} 
The left equation is clear by definition. 
Since $P\in M_i(\cA)$ is contained in $(i+1)$ lines of $\cA$, 
the middle equation holds. 
Finally, regarding all the intersection points of $\cA$ as 
the concentrations of the intersections of 2 lines of $\cA$, 
we have the right equation. 
\end{proof} 
Now we determine all the possibilities of $F(\cA)$ 
for $\cA\in\cS_\ell$ $(\ell\leq12)$. 
\begin{prop}\label{Classify} 
Let $\ell\in\bZ_{\leq12}$ and 
$\cA\in\cS_{\ell}$ with $\exp(\cA)=(1,a,b)$. 
Then we have 
$$(\ell,\min(a,b),F(\cA))\in\left\{ 
\begin{array}{ccc} 
(9, 4, [0, 12]), 
& 
(11, 5,[1, 14, 2]), 
& 
(11, 5,[4, 11, 3]), 
\\ 
(11, 5,[7, 8, 4]), 
& 
(11, 5,[10, 5, 5]), 
& 
(12, 5,[0, 16, 3]) 
\end{array} 
\right\}. 
$$ 
In particular, we have 
$\cS_{\ell}=\emptyset$, 
$\cF_{\ell}=\cI_{\ell}$ for $2\leq\ell\leq 8$ 
and $\cS_{10}=\emptyset$. 
\end{prop} 
\begin{proof} 
Note that $b=\ell-1-a$. 
We may assume $a\leq(\ell-1)/2$. 
By Lemma \ref{a-2}, we may assume 
$\ell\geq7$ and $F_i=0$ for $i\geq a-1$. 
Since $\chi({\cA},t)=(t-1)(t-a)(t-b)$ 
by Theorem \ref{factorization}, we have 
$\mu_{\cA}=ab+\ell-1=(\ell-1)(a+1)-a^2$. 
Also, since $\cA\in\cS_{\ell}$, we have 
$\sum_{H\in\cA}n_{\cA,H}\leq a\ell$. 
Thus we have the inequalities as follows 
$$ 
\sum_{i=1}^{a-2}iF_i(\cA)=(\ell-1)(a+1)-a^2, 
\quad 
\sum_{i=1}^{a-2}(i+1)F_i(\cA)\leq a\ell, 
\quad 
\sum_{i=1}^{a-2} 
\binom{i+1}{2}F_i(\cA)=\binom{\ell}{2}. 
$$ 
Solving above inequalities under the condition 
$0\leq a\leq(\ell-1)/2$ and $7\leq\ell\leq12$, we obtain only 
6 triplets $[\ell,a,F]$ appearing in the right hand side 
of the statement. Now $\cS_\ell=\emptyset$ for $2\leq\ell\leq8$ 
and $\cS_{10}=\emptyset$ 
are clear.  By Lemma \ref{Red}, we have 
$\cF_\ell=\cI_\ell$ for $2\leq \ell\leq8$. 
\end{proof} 
\begin{defn} 
For $H\in\cA$ and $i\in\bZ_{>0}$, we set 
$\mu_{\cA,H}=\sum_{P\in H}\mu_P(\cA)$ and 
$$ 
M_{i}(H,\cA)=M_i(\cA)\cap H,\ 
F_{H,i}(\cA)=\left|M_{i}(H,\cA)\right|,\ 
F_H(\cA)=[F_{H,1}(\cA),F_{H,2}(\cA),\ldots]. 
$$ 
\end{defn} 
\begin{lem} 
For $H\in\cA$, the invariant $F_{H}(\cA)$ satisfies 
the following formulae. 
$$ 
\sum_iF_{H,i}(\cA)=n_{\cA,H},\ 
\sum_iiF_{H,i}(\cA)=\mu_{\cA,H}=|\cA|-1,\ 
\sum_{H\in\cA}F_{H,i}(\cA)=(i+1)F_i(\cA). 
$$ 
\end{lem} 
\begin{proof} 
The formulae above are clear by definitions and the fact that 
$|\cA_P|=\mu_P(\cA)+1$. 
\end{proof} 
\medskip 
 
In the following sections, 
we determine an arrangement $\cA\in\cS_{\ell}$ for $\ell\leq12$. 
Namely, we determine the lattice structures of $\cA$ 
up to the permutations ${\mathfrak S}_{\ell}$ of 
indices of hyperplanes and 
determine their realizations 
in $\bP_{\bC}^2$ up to the action of $\on{PGL}(3,\bC)$. 
 
The hyperplanes of $\cA$ are denoted by $\cA=\{H_1,\ldots,H_\ell\}$, 
while the defining equation of each $H_i$ is denoted by $h_i$. 
The intersection points of $\cA$ 
satisfying $\cA_P=\{H_{a_i}\mid i\in I\}$ 
is denoted by $\{a_i\mid i\in I\}$. 
The line passing through $P$ and $Q$ is 
denoted by $\overline{PQ}$. 
For the coordinate calculation, 
we regard $\bP_{\bC}^2$ as 
the union of affine part ${\bC}^2$ and 
the infinity line $H_{\infty}$. 
\end{section} 
\begin{section}{Determination of $\cS_9$}\label{S9} 
In this section, we show that $\cS_9$ consists of 
dual Hesse arrangements. 
\begin{subsection}{Lattice structure of $\cA\in\cS_9$}\label{S9L} 
We determine the lattice of $\cA\in\cS_9$. 
By Proposition \ref{Classify}, we have $F(\cA)=[0,12]$. 
Note that $F_{H}(\cA)=[0,4]$ for any $H\in\cA$ 
since $F(\cA)=[0,12]$ and $\sum_iiF_{H,i}(\cA)=\ell-1=8$. 
Concerning $M_2(H_9,\cA)$, we may set 
$$\{1,2,9\},\{3,4,9\},\{5,6,9\},\{7,8,9\} 
\in M_2(\cA).$$ 
Since $H_1\cap H_3$ lies on 
$H_5$, $H_6$, $H_7$ or $H_8$, 
we may set 
$\{1,3,5\}\in M_2(\cA)$ by symmetry. 
Since $H_1\cap H_7\neq H_1\cap H_8$, 
they are other 2 points of $M_2(H_1,\cA)$. 
Thus we may set 
$\{1,4,7\},\{1,6,8\}\in M_2(H_1,\cA)$ by symmetry 
of $(3,5)(4,6)$. Namely, 
$$\{1,3,5\},\{1,4,7\},\{1,6,8\}\in M_2(\cA).$$ 
Investigating 
$M_2(H_3,\cA)$, 
$M_2(H_4,\cA)$ and 
$M_2(H_2,\cA)$, 
it is easy to see 
$$ 
\{2,3,8\},\{3,6,7\}, 
\{2,4,6\},\{4,5,8\}, 
\{2,5,7\} 
\in M_2(\cA). 
$$ 
Now we obtain all the points of $M_2(\cA)$, thus 
the lattice structure of $\cA$ is determined. 
\end{subsection} 
\begin{subsection}{Realization of $\cA\in\cS_9$}\label{S9R} 
We determine the realization of $\cA\in\cS_9$ in $\bP_{\bC}^2$. 
We may set $H_9$ as the infinity line $H_\infty$, 
$h_1=x$, $h_2=x-1$, $h_3=y$, $h_4=y-1$ and 
$$ 
\{1,6,8\}=(0,p), 
\{2,5,7\}=(1,q), 
\{3,6,7\}=(r,0), 
\{4,5,8\}=(s,1) 
\quad 
(p,q,r,s\neq0,1). 
$$ 
Note that 
$(0,0),(1,q),(s,1)\in H_5$, 
$(1,1),(0,p),(r,0)\in H_6$, 
$(0,1),(1,q),(r,0)\in H_7$ and 
$(1,0),(0,p),(s,1)\in H_8$. 
Therefore we have 
$$sq=(1-r)(1-p)=r(1-q)=p(1-s)=1.$$ 
Solving these equations, we have 
$$(p,q,r,s)=(-\omega^2,-\omega,-\omega,-\omega^2),$$ 
where $\omega$ is a primitive third root of unity, 
and $h_i$ for $5\leq i\leq 8$ as follows. 
$$ 
h_5=y+\omega x,\ 
h_6=y+\omega x+\omega^2,\ 
h_7=y-\omega^2 x-1,\ 
h_8=y-\omega^2 x+\omega^2. 
$$ 
By this construction, for the permutation 
$\sigma\in{\mathfrak S}_{9}^\ast=\left\{\sigma\in{\mathfrak S}_{9} 
\mid \sigma(L(\cA))=L(\cA)\right\}$ preserving the 
lattice, there exists a $\on{GL}(3,\bC)$-action 
sending each $H_i$ to $H_{\sigma(i)}$, 
or sending each $H_i$ to $\overline{H_{\sigma(i)}}$, 
where $\overline{H_{i}}$ stands for the Galois conjugate 
of $H_i$ by $\on{Gal}(\bQ[\sqrt{-3}]/\bQ)$. 
Note also that $\cA$ is transferred to 
$\overline{\cA}$ 
by 
$\left[(x,y,z)\mapsto(y,x,z)\right]\in\on{GL}(3,\bC)$, 
which sends $H_i$ to $\overline{H_{\mu(i)}}$ 
where 
$\mu=(1,3)(2,4)(7,8)\in{\mathfrak S}_{9}^\ast$. 
Thus $\cA$ is realized uniquely up to the $\on{GL}(3,\bC)$-action. 
\end{subsection} 
\begin{subsection}{Verifying $\cA\in\cS_9\subset\cR_9$}\label{S9V} 
We check the freeness of $\cA$ realized in \S \ref{S9R} and 
show that $\cA\in\cR_9$. 
We set $\cA_1=\cA\cup\{H_{10}\}$ where $h_{10}=x-y$. 
Then we have 
$$\cA_1\cap H_{10}=\{ 
(0,0),\ 
(1,1),\ 
((1-\omega^2)^{-1},(1-\omega^2)^{-1}),\ 
((1-\omega)^{-1},(1-\omega)^{-1}),\ 
H_{10}\cap H_{\infty}\}. 
$$ 
Since $\mu_{\cA_1}=\mu_{\cA}+5$, we have 
$\chi({\cA_1},t)=(t-1)(t-4)(t-5)$. 
By Theorem \ref{ABT}, we have 
$\cA_1\in\cF_{10}$ with $\exp(\cA_1)=(1,4,5)$, 
and hence $\cA\in\cS_9$. 
We set $\cA_2=\cA_1\setminus\{H_{\infty}\}$ and 
$\cA_3=\cA_2\setminus\{H_7\}$. 
Since $n_{\cA_1,H_{\infty}}=5$, we have 
$\cA_2\in\cF_9$ with $\exp(\cA_2)=(1,4,4)$. 
Since $n_{\cA_2,H_7}=5$, we have 
$\cA_3\in\cF_8=\cI_8$. Therefore 
$\cA_2\in\cI_9$, $\cA_1\in\cI_{10}$ and 
$\cA\in\cR_9$. 
 
\medskip 
 
In fact, to check whether $\cA\in\cF_9$ belongs to $\cS_9$ or not, 
we have only to check $F(\cA)$. 
\begin{lem} 
If $\cA\in\cF_9$ satisfies $F(\cA)=[0,12]$, 
then $\cA\in\cS_9$. 
\end{lem} 
\begin{proof} 
For any $H\in\cA$, 
since $9-1=\mu_{\cA,H}=2n_{\cA,H}$, we have 
$n_{\cA,H}=4$. 
\end{proof} 
\begin{defn}\label{dH} 
An arrangement in $\bC^3$ 
is called a {\it dual Hesse} arrangement 
if it is $\on{GL}(3,\bC)$-equivalent to 
$(\varphi_{\on{dH}}=0)$, where 
$$ 
\varphi_{\on{dH}}= 
(x^3-y^3)(y^3-z^3)(z^3-x^3). 
$$ 
\end{defn} 
It is easy to see that 
$\cA=(\varphi_{\on{dH}}=0)$ 
satisfies $F(\cA)=[0,12]$. 
Therefore, 
$$\cS_9=\{\text{dual Hesse arrangements}\}\subset\cR_9.$$ 
\end{subsection} 
\begin{subsection}{Addition to $\cA\in\cS_9$} 
The structures of $\cF_9$ and $\cF_{10}$ are given as below. 
\begin{prop}\label{9-10} 
\item{(1)} 
Let $H\in \cA_1\in\cF_{10}$ such that 
$\cA=\cA_1\setminus\{H\}\in\cS_9$. 
Then, $\cA_1\in\cI_{10}$ and $\cA_1$ is unique up to 
the $\on{GL}(3,\bC)$-action. 
\item{(2)} 
$\cF_{9}=\cR_{9}=\cI_9\sqcup\cS_9$ and 
$\cF_{10}=\cI_{10}$. 
\end{prop} 
\begin{proof} 
\item{(1)} 
We may assume $\cA$ has the description 
as in \S \ref{S9L} and \S \ref{S9R}. 
By Theorem \ref{ad}, we have $n_{\cA_1,H}=5$ and hence 
$F_{H}(\cA_1)=[3,0,2]$. 
Since $H\not\in\cA$, $H$ is one of the following. 
\begin{eqnarray*} 
&& 
\overline{\{1,2,9\}\{3,6,7\}}, 
\overline{\{1,2,9\}\{4,5,8\}}, 
\overline{\{1,3,5\}\{2,4,6\}}, 
\overline{\{1,3,5\}\{7,8,9\}}, 
\\ && 
\overline{\{1,4,7\}\{2,3,8\}}, 
\overline{\{1,4,7\}\{5,6,9\}}, 
\overline{\{1,6,8\}\{2,5,7\}}, 
\overline{\{1,6,8\}\{3,4,9\}}, 
\\ && 
\overline{\{2,3,8\}\{5,6,9\}}, 
\overline{\{2,4,6\}\{7,8,9\}}, 
\overline{\{2,5,7\}\{3,4,9\}}, 
\overline{\{3,6,7\}\{4,5,8\}}. 
\end{eqnarray*} 
Recall that any $\sigma\in{\mathfrak S}_9^\ast$ 
is realized by the action of $\on{GL}(3,\bC)$ and 
$\on{Gal}(\bQ[\sqrt{-3}]/\bQ)$. 
Therefore it suffices to show that 
${\mathfrak S}_9^\ast$ acts transitively 
on the pairs in the above list. 
 
Observe that each points 
of $M_2(\cA)$ lies on two candidates of $H$. 
We denote 
the ${\mathfrak S}_9^\ast$-equivalence 
by the symbol ``$\sim$''. 
First note that 
$(\{1,2,9\},\{3,6,7\})\sim (\{3,6,7\},\{1,2,9\})$ 
by $(1,3)(2,6)(7,9)\in{\mathfrak S}_9^\ast$ and 
$(\{1,2,9\},\{3,6,7\})\sim (\{1,2,9\},\{4,5,8\})$ 
by $(3,4)(5,7)(6,8)\in{\mathfrak S}_9^\ast$. 
As the point transferred from $\{1,2,9\}$ has the 
same property as above, it follows that 
$(\{1,2,9\},\{3,6,7\})\sim (\{3,6,7\},\{4,5,8\})$. 
Namely, we have 
$$ 
(\{1,2,9\},\{3,6,7\})\sim 
(\{1,2,9\},\{4,5,8\})\sim 
(\{3,6,7\},\{4,5,8\}). 
$$ 
By applying 
$(2,3)(4,7)(5,9), (2,7)(4,9)(6,8), 
(2,6)(3,5)(8,9)\in{\mathfrak S}_9^\ast$, 
We have 
\begin{eqnarray*} 
(\{1,2,9\},\{3,6,7\}) 
\sim (\{1,3,5\},\{2,4,6\}) 
\sim (\{1,4,7\},\{2,3,8\}) 
\sim (\{1,6,8\},\{2,5,7\}). 
\end{eqnarray*} 
Therefore we have the following, which completes the proof 
of the uniqueness of $\cA_1$. 
\begin{eqnarray*} 
(\{1,2,9\},\{3,6,7\})&\sim& 
(\{1,3,5\},\{2,4,6\})\sim 
(\{1,3,5\},\{7,8,9\})\sim 
(\{2,4,6\},\{7,8,9\}) 
\\ 
&\sim& 
(\{1,4,7\},\{2,3,8\})\sim 
(\{1,4,7\},\{5,6,9\})\sim 
(\{2,3,8\},\{5,6,9\}) 
\\ 
&\sim& 
(\{1,6,8\},\{2,5,7\})\sim 
(\{1,6,8\},\{3,4,9\})\sim 
(\{2,5,7\},\{3,4,9\}). 
\end{eqnarray*} 
Note that $H_{10}$ in \S \ref{S9V} is $\overline{\{1,3,5\}\{2,4,6\}}$ 
and $\cA\cup\{H_{10}\}\in\cI_{10}$.  By the uniqueness of 
$\cA_1$, we conclude that $\cA_1\in\cI_{10}$. 
Therefore (1) is verified. 
\item{(2)} 
Since $\cF_8=\cI_8$ and $\cS_9\subset\cR_9$, we have 
$\cF_9=\cI_9\sqcup\cS_9=\cR_9$ by Lemma \ref{Red}. 
Let $\cA\in\cF_{10}$.  Since $\cS_{10}=\emptyset$, there exists 
$H\in\cA$ such that $\cA'=\cA\setminus\{H\}\in\cF_9=\cI_9\sqcup\cS_9$. 
If $\cA'\in\cI_9$, then $\cA\in\cI_{10}$. 
If $\cA'\in\cS_9$, we also have $\cA\in\cI_{10}$ by (1). 
Therefore we have $\cF_{10}=\cI_{10}$. 
\end{proof} 
We remark that now Theorem \ref{Main} is 
established for $|\cA|\leq10$ 
by Propositions \ref{Classify} and \ref{9-10}. 
We give the proof of Corollary \ref{le4}. 
\begin{proof}[Proof of Corollary \ref{le4}] 
The proof is by the induction on $\ell=|\cA|$.  If $\ell\leq10$, we 
have nothing to prove. Assume that $\ell\geq11$. 
If $\cA\in\cF_{\ell}\setminus\cS_{\ell}$, then 
$H\in\cA$ such that $\cA'=\cA\setminus\{H\}\in\cF_{\ell-1}$. 
Since $\exp(\cA')=(1,a-1,b)$ or $(1,a,b-1)$, 
we have $\cA'\in\cI_{\ell-1}$ by induction hypothesis, 
and hence $\cA\in\cI_{\ell}$. Thus we may assume 
$\cA\in\cS_{\ell}$. We set $a\leq b$ and take $H\in\cA$. 
By Lemma \ref{a-2}, we have $\mu_P(\cA)\leq a-2$ for any $P\in H$. 
By definition of $S_{\ell}$, we have $n_{\cA,H}\leq a$. 
However it is a contradiction since we have the following 
inequalities. 
$$11-1\leq\ell-1=\mu_{\cA,H}\leq (a-2)a 
\leq(4-2)\cdot4=8. 
\eqno\qed$$ 
\renewcommand{\qed}{} 
\end{proof} 
\end{subsection} 
\end{section} 
\begin{section}{Determination of $\cS_{11}$}\label{S11} 
In this section, we show that $\cS_{11}$ consists of pentagonal 
arrangements. 
\begin{subsection}{Absence of $\cA\in\cS_{11}$ with 
$F(\cA)=[1,14,2]$} 
Let $\cA\in\cS_{11}$. 
By Proposition \ref{Classify}, we have 
$F(\cA)=[1,14,2]$, 
$[4,11,3]$, $[7,8,4]$ or $[10,5,5]$. 
 
Suppose $F(\cA)=[1,14,2]$. Take $P\in M_{3}(\cA)$. 
Since $|M_1(\cA)\cup M_{3}(\cA)\setminus\{P\}|=2$ 
and $|\cA_P|=4$, there exists 
$H\in\cA_P$ such that 
$M_1(H,\cA)=\emptyset$ and $M_3(H,\cA)=\{P\}$. 
Then it follows that 
$11-1=\mu_{\cA,H}=0+2F_{H,2}(\cA)+3\cdot1$, 
which is impossible. 
 
Therefore $F(\cA)=[4,11,3], [7,8,4]$ or $[10,5,5]$. 
In the following subsections, we 
show that only the case $F(\cA)=[10,5,5]$ 
occurs, which corresponds to the case when 
$\cA$ is a pentagonal arrangement. 
\end{subsection} 
\begin{subsection}{Subarrangement $\cA'$ of $\cA$} 
In the case $F(\cA)=[4,11,3]$ or $[7,8,4]$, 
we construct a subarrangement 
$\cA'=\{H_1,\ldots,H_{10}\}$ of $\cA$ satisfying the following. 
$$ 
F(\cA')=[9,6,3],\quad 
n_{\cA',H_i}= 
\begin{cases} 
4 & i=1 \\ \leq5 & 2\leq i\leq 10 
\end{cases},\quad 
M_3(\cA')=\left\{\begin{array}{ccc} 
H_1\cap H_2, \\ H_1\cap H_3, \\ H_2\cap H_3 
\end{array}\right\}. 
\eqno{(\ast)} 
$$ 
 
Suppose $F(\cA)=[4,11,3]$. 
Note that $n_{\cA,H}=4,5$ for any $H\in\cA$, since 
$3\cdot3<10=\mu_{\cA,H}$. 
Thus $n_{\cA,H}=4,5$ for any $H\in\cA$. Since 
$\sum_{H\in\cA}n_{\cA,H}=2\cdot4+3\cdot11+4\cdot3 
=5\cdot11-2$, we may set $n_{\cA,H_1}=n_{\cA,H_2}=4$ and 
$n_{\cA,H_i}=5$ for $3\leq i\leq 11$. 
Note that $F_{H_i,3}(\cA)\geq2$ for $i=1,2$, since 
$2\cdot3+3\cdot1<10=\mu_{\cA,H}$. 
Thus we may set $M_3(\cA)=\{P_1,P_2,P_3\}$, 
$H_1=\overline{P_2P_3}$ and $H_2=\overline{P_1P_3}$. 
Since $n_{\cA,H_1}=4=\cA_{P_1}\cap H_1$, we have 
$\overline{P_1P_2}\in\cA$, which we set $H_3$. 
Since $|\bigcup_{P\in M_3(\cA)}\cA_{P}|=4\cdot3-3=9<11$, 
we may set $H_{11}\cap M_3(\cA)=\emptyset$. 
Then $F_{\cA,H_{11}}=[0,5,0]$. 
Now it is easy to check that 
$\cA'=\cA\setminus\{H_{11}\}$ 
satisfies the condition $(\ast)$. 
 
Suppose $F(\cA)=[7,8,4]$. 
Since $\sum_{H\in\cA}n_{\cA,H}=2\cdot7+3\cdot8+4\cdot4=5\cdot11-1$, 
we may assume $n_{\cA,H_1}=4$ and $n_{\cA,H_i}=5$ for $2\leq i\leq 11$. 
Note that $F_{H_1,3}(\cA)\neq1,4$ since 
$2\cdot3+3\cdot1<\mu_{\cA,H_1}=10<3\cdot4$. 
Thus $F_{H_1,3}(\cA)=2, 3$.  We set 
$M_3(\cA)=\{P_1, P_2, P_3, P_4\}$ so that 
$P_1\not\in H_1=\overline{P_2P_3}$. 
Since $|\cA_{P_1}\cap H_1|=4=n_{\cA,H_1}$, 
we have $P_2,P_3\in\cA_{P_1}$. Therefore 
we may set $H_2=\overline{P_1P_3}$ and 
$H_3=\overline{P_1P_2}$. 
Since $|\cA_{P_4}|=4$, we may set 
$M_3(\cA)\cap H_{11}=\{P_4\}$. 
Since $n_{\cA,H_{11}}=5$ and $F_{H_{11},3}(\cA)=1$, we have 
$F_{H_{11}}(\cA)=[1,3,1]$. 
Now it is easy to check that 
$\cA'=\cA\setminus\{H_{11}\}$ 
satisfies the condition $(\ast)$. 
\end{subsection} 
\begin{subsection}{Lattice structure of $\cA'$} 
First we determine $F_{H_i}(\cA')$ for $1\leq i\leq 10$. 
We have 
$n_{\cA',H}=4,5$ for any $H\in\cA'$ 
since $|M_3(H,\cA')|\leq2$ and 
$2+3\cdot2<9=\mu_{\cA',H}$. 
Since $n_{\cA',H_1}=4$ and $F_{3,H_1}(\cA')=2$, 
we have  $F_{H_1}(\cA')=[1,1,2]$. 
Since $\bigcup_{P\in M_3(\cA')}\cA_P'=4\cdot3-3=9$, 
we may set $M_{3}(H_{10},\cA')=\emptyset$, 
and hence $F_{H_{10}}=[1,4,0]$. 
Since $H_2\cap H_{10}\neq H_3\cap H_{10}$, 
we may set $H_2\cap H_{10}\in M_2(\cA')$. 
Since $F_{H_2,3}(\cA')=2$, we have 
$F_{H_2}(\cA')=[1,1,2]$. 
Since $\sum_{H\in\cA'}n_{\cA',H}=2\cdot9+3\cdot6+4\cdot3=5\cdot10-2$, 
we have $n_{\cA',H_i}=4$ for $i=1,2$ and 
$n_{\cA',H_i}=5$ for $3\leq i\leq 10$. 
We have $F_{H_3}(\cA')=[3,0,2]$ and 
$F_{H_i}(\cA')=[2,2,1]$ for $4\leq i\leq 9$ 
since $F_{H_3,3}(\cA')=2$ and $F_{H_i,3}(\cA')=1$. 
As a conclusion, we have the following. 
$$ 
F_{H_i}(\cA')= 
[1,1,2]\ (i=1,2),\ 
[3,0,2]\ (i=3),\ 
[2,2,1]\ (4\leq i\leq 9),\ 
[1,4,0]\ (i=10). 
$$ 
 
Now we determine the lattice structure of $\cA'$. We may set 
$$ 
M_3(\cA')=\{ 
P_1=\{2,3,4,5\}, 
P_2=\{1,3,6,7\}, 
P_3=\{1,2,8,9\} 
\}. 
$$ 
Note that 
$\{3,8\},\{3,9\},\{3,10\}\in M_1(\cA')$ 
since $F_{H_3}(\cA')=[3,0,2]$. 
By symmetry of $(4,5)$ or $(6,7)$, 
we may set 
$\{1,4,10\}, \{2,6,10\}\in M_2(\cA')$. 
Since $H_8\cap H_{10}$ lies on $H_5$ or $H_7$, 
we may set $\{5,8,10\}\in M_2(\cA')$ 
by symmetry of $(1,2)(4,6)(5,7)$. 
We also have $\{7,9,10\}\in M_2(\cA')$. 
Since 
$\{\{7,9,10\}\}\cup\left(\cA'_{P_1}\cap H_7\right)$ 
defines all intersection points on $H_7$, 
we have $H_7\cap H_8\in M_2(\cA)$. 
By the same reasoning for 
$\{\{5,8,10\}\}\cup\left(\cA'_{P_2}\cap H_8\right)$ 
on $H_8$, we have $H_4\cap H_8\in M_2(\cA)$. 
Thus we have $M_2(H_8,\cA')=\{\{5,8,10\},\{4,7,8\}\}$. 
Since 
$F_{H_i,2}(\cA')=2$ for $i=5,6,9$, 
The last point of $M_2(\cA)$ is $\{5,6,9\}$. 
Therefore $M_2(\cA')$, and hence $M_1(\cA')$, are as follows, 
which determine the lattice of $\cA'$. 
\begin{eqnarray*} 
M_2(\cA')&=&\left\{ 
\{1,4,10\},\ 
\{2,6,10\},\ 
\{4,7,8\},\ 
\{5,6,9\},\ 
\{5,8,10\},\ 
\{7,9,10\} 
\right\}, 
\\ 
M_1(\cA')&=&\left\{ 
\{1,5\},\ 
\{2,7\},\ 
\{3,8\},\ 
\{3,9\},\ 
\{3,10\},\ 
\{4,6\},\ 
\{4,9\},\ 
\{5,7\},\ 
\{6,8\} 
\right\}. 
\end{eqnarray*} 
\end{subsection} 
\begin{subsection}{Realization of $\cA'$} 
We determine the realization of $\cA'$ in $\bP_{\bC}^2$. 
We may set $H_{10}$ as the infinity line $H_{\infty}$, 
$P_1=(1,0)$, $P_2=(0,1)$ and $P_3=(0,0)$. 
Then 
$$ 
h_1=x,\ 
h_2=y,\ 
h_3=x+y-1,\ 
h_4=x-1,\ 
h_6=y-1. 
$$ 
Set $\{4,7,8\}=(1,p)$ and $\{5,6,9\}=(q,1)$ 
($p,q\neq0$). 
Then we have 
$$ 
h_5=x-(q-1)y-1 
,\ 
h_7=(p-1)x-y+1 
,\ 
h_8=px-y 
,\ 
h_9=x-qy. 
$$ 
Since 
$H_5\parallel H_8$ and 
$H_7\parallel H_9$, 
we have 
$p(q-1)=(p-1)q=1$. 
Therefore we conclude that 
$p=q=\zeta$ where $\zeta$ is a solution of 
$\zeta^2-\zeta-1=0$, and we may reset the equations as 
$$ 
h_5=\zeta x-y-\zeta 
,\ 
h_7=x-\zeta y+\zeta 
,\ 
h_8=\zeta x-y 
,\ 
h_9=x-\zeta y. 
$$ 
\end{subsection} 
\begin{subsection}{Absence of $\cA\in\cS_{11}$ 
with $F(\cA)=[4,11,3]$ or $[7,8,4]$} 
We show that we cannot extend the realization of $\cA'$ 
obtained above to $\cA$. 
Assume that $\cA=\cA'\cup\{H_{11}\}$ is realizable. 
 
Suppose $F(\cA)=[4,11,3]$. 
Recall that $F_{H_{11}}(\cA)=[0,5,0]$. 
Since $|M_1(\cA')\cap H_{11}|=5$ 
and 
$M_1(\cA')\subset 
\left\{\{1,5\},\{4,9\}\right\} 
\cup \bigcup_{i=3,6,7}H_i$, 
we have 
$H_{11}=\overline{\{1,5\}\{4,9\}} 
=\overline{(0,-\zeta)(1,\zeta^{-1})}$ and 
$h_{11}=(1-2\zeta)x+y+\zeta$. 
Therefore $H_{10}\cap H_{11}\in M_1(\cA)$, 
a contradiction. 
 
Suppose $F(\cA)=[7,8,4]$. 
Recall that $F_{H_{11}}(\cA)=[1,3,1]$. 
Thus $M_1(H_{11},\cA)=\{H_i\cap H_{11}\}$ for some $1\leq i\leq 10$. 
Since 
$n_{\cA',H_i}+1=n_{\cA,H_i}\leq5$, 
we have $i=1$ or $2$. We may set $H_{1}\cap H_{11}\in M_1(\cA)$ 
by the symmetry of the coordinates $x$ and $y$. Note that 
$|M_1(\cA')\cap H_{11}|=3$ and $|M_2(\cA')\cap H_{11}|=1$. 
In particular, $H_2\cap H_{11}=\{2,7\}$ or $\{2,6,10\}$. 
 
Assume that $H_2\cap H_{11}=\{2,7\}=(-\zeta,0)$. 
Then $M_2(\cA')\cap H_{11}=\{\{5,6,9\}\}$ or $\{\{5,8,10\}\}$. 
If $H_{11}\ni\{5,6,9\}=(\zeta,1)$, we have $h_{11}=x-2\zeta y+\zeta$. 
Therefore $H_{10}\cap H_{11}\in M_1(\cA)$, a contradiction. 
If $H_{11}\ni\{5,8,10\}$, 
we have $H_{11}\cap M_1(\cA')=\left\{ 
\{2,7\},\{3,9\},\{4,6\}\right\}$. 
Since $\{4,6\}=(1,1)$, 
we have $h_{11}=x-(\zeta+1)y+\zeta$, 
which contradicts to $H_{11}\parallel H_{8}$. 
 
Assume that $H_2\cap H_{11}=\{2,6,10\}$. Then 
we have 
$M_1(\cA')\cap H_{11}= 
\left\{\{3,8\},\{4,9\},\{5,7\}\right\}$. 
Since 
$\{4,9\}=(1,\zeta-1)$ 
and 
$\{5,7\}=(\zeta+1,\zeta+1)$ 
we have $H_{11}\not\parallel H_{2}$, 
a contradiction. 
 
\medskip 
 
Now we may assume that $F(\cA)=[10,5,5]$. 
\end{subsection} 
\begin{subsection}{Lattice structure of $\cA\in\cS_{11}$}\label{S11L} 
We determine the lattice of $\cA\in\cS_{11}$. 
First we show that 
$\overline{PQ}\in\cA$ for any $P,Q\in M_3(\cA)$, $P\neq Q$. 
Assume that there exist $P,Q\in M_3(\cA)$ such that 
$\overline{PQ}\not\in\cA$. Note that 
$\cA$ has $10+5+5=20$ intersection points, 
and $\cA_P\cup\cA_Q$ covers $4\cdot4+2=18$ of them. 
We set the left 2 intersection points in 
$\cA\setminus(\cA_P\cup\cA_Q)$ as $T_1$ and $T_2$. 
If $H=\overline{T_1T_2}\in\cA$, then 
$\{T_1,T_2\}\cap(\cA_P\cap H)\neq\emptyset$ 
since $|\cA_P\cap H|=4$ and $n_{\cA,H}\leq5$. 
However, it contradicts to the choice of $T_i$. 
If $\overline{T_1T_2}\not\in\cA$, then 
$(\cA_{P}\cup\cA_{Q})\cap(\cA_{T_1}\cup\cA_{T_2}) 
\neq\emptyset$ since $|\cA|=11$, 
$|\cA_{P}\cup\cA_{Q}|=8$ and 
$|\cA_{T_1}\cup\cA_{T_2}|\geq4$. 
It also contradicts to the choice of $T_i$. 
 
Next we determine $F_{H_i}(\cA)$ for $1\leq i\leq 10$. 
Note that $n_{\cA,H}=5$ for $H\in\cA$ since 
$\sum_{H\in\cA}n_{\cA,H}=2\cdot10+3\cdot5+4\cdot5=5\cdot11$. 
We also have $F_{H,3}(\cA)\leq2$ for $H\in\cA$ since 
$\mu_{\cA,H}=10<1\cdot2+3\cdot3$. 
It follows that, for $P,Q\in M_3(\cA)$ with 
$P\neq Q$, $\overline{PQ}\in\cA$ are distinct each other, forming 
$\binom{5}{2}=10$ lines of $\cA$. 
Thus we may assume $F_{H_i,3}(\cA)=2$, i.e., 
$F_{H_i}(\cA)=[2,1,2]$, for $1\leq i\leq 10$. 
Since $4\cdot5=\sum_{H\in\cA}F_{H,3}(\cA) 
=2\cdot10+F_{H_{11},3}(\cA)$, 
we have $F_{H_{11},3}(\cA)=0$, i.e., 
$F_{H_{11}}(\cA)=[0,5,0]$. 
Therefore, we have 
$$F_{H_i}(\cA)=[2,1,2]\ (1\leq i\leq 10), 
\quad F_{H_{11}}(\cA)=[0,5,0].$$ 
 
We investigate the lattice structure of $\cA$. 
We may set $M_3(\cA)=\{P_i\mid 1\leq i\leq5\}$ and 
\begin{eqnarray*} 
&& 
H_1=\overline{P_1P_2},\ 
H_2=\overline{P_1P_3},\ 
H_3=\overline{P_1P_4},\ 
H_4=\overline{P_1P_5},\ 
H_5=\overline{P_2P_3},\ 
\\ && 
H_6=\overline{P_2P_4},\ 
H_7=\overline{P_2P_5},\ 
H_8=\overline{P_3P_4},\ 
H_9=\overline{P_3P_5},\ 
H_{10}=\overline{P_4P_5}, 
\end{eqnarray*} 
or, in other words, $M_3(\cA)$ consists of 
the following five points. 
$$ 
P_1=\{1,2,3,4\},\ 
P_2=\{1,5,6,7\},\ 
P_3=\{2,5,8,9\},\ 
P_4=\{3,6,8,10\},\ 
P_5=\{4,7,9,10\}. 
$$ 
Since $H_1\cap H_{11}\in M_2(\cA)$ lies on $H_8$, $H_9$ or $H_{10}$, 
we may set $\{1,9,11\}\in M_2(\cA)$ by symmetry. 
Since $H_3\cap H_{11}\in M_2(\cA)$ lies on $H_5$ or $H_7$, 
we may set $\{3,5,11\}\in M_2(\cA)$ by symmetry of 
$(2,4)(5,7)(8,10)$. Investigating 
$H_{10}\cap H_{11}, H_{6}\cap H_{11}, H_{8}\cap H_{11} \in M_2(\cA)$ 
in this order, we have 
$\{2,10,11\},\{4,6,11\},\{7,8,11\}\in M_2(\cA)$. 
Thus $M_2(\cA)$ is determined. 
$$ 
M_2(\cA)=\left\{ 
\{1,9,11\},\ 
\{2,10,11\},\ 
\{3,5,11\},\ 
\{4,6,11\},\ 
\{7,8,11\} 
\right\}. 
$$ 
Now $M_2(\cA)$ and $M_3(\cA)$ are determined, which 
gives the lattice structure of $\cA$. 
\end{subsection} 
\begin{subsection}{Realization of $\cA\in\cS_{11}$}\label{S11R} 
We determine the realization of $\cA\in\cS_{11}$ 
in $\bP_{\bC}^2$. 
We may set $H_{11}$ as the infinity line $H_{\infty}$, 
$P_1=(0,1)$, $P_2=(0,0)$ and $P_3=(1,0)$. 
By definition of $H_1$, $H_2$, $H_5$ and 
the fact that $P_3\in H_9\parallel H_1$ and 
$P_1\in H_3\parallel H_5$ imply that 
$$h_1=x,\ h_2=x+y-1,\ h_3=y-1,\ h_5=y,\ h_9=x-1.$$ 
We set $P_4=(p,1)$ and $P_5=(1,q)$. 
Since $H_{10}=\overline{P_4P_5}\parallel H_2$ and 
$H_{4}=\overline{P_1P_5}\parallel 
H_{6}=\overline{P_2P_4}$, 
we have $p=q$ and $p(q-1)=1$. 
Thus we have $p=q=\zeta$ where $\zeta$ is a solution of 
$\zeta^2-\zeta-1=0$. 
The left defining equations $h_i$ of $H_i$ are as follows. 
$$ 
h_4=x-\zeta y+\zeta,\ 
h_6=x-\zeta y,\ 
h_7=\zeta x-y,\ 
h_8=\zeta x-y-\zeta,\ 
h_{10}=x+y-\zeta-1. 
$$ 
By this construction, for the permutation 
$\sigma\in{\mathfrak S}_{11}^{\ast} 
=\left\{\sigma\in{\mathfrak S}_{11} 
\mid \sigma(L(\cA))=L(\cA)\right\}$, 
there exists a $\on{GL}(3,\bC)$-action 
sending each $H_i$ to $H_{\sigma(i)}$, 
or sending each $H_i$ to $\overline{H_{\sigma(i)}}$, 
where $\overline{H_{i}}$ stands for the Galois conjugate 
of $H_i$ by $\on{Gal}(\bQ[\sqrt{5}]/\bQ)$. 
Note also that $\cA$ is transferred to $\overline{\cA}$ 
by 
$\left[(x,y,z)\mapsto(\zeta x+y,x+\zeta y,z) 
\right]\in\on{GL}(3,\bC)$, 
which sends $H_i$ to $\overline{H_{\nu(i)}}$ 
where 
$\nu=(1,6,5,7)(2,10)(3,8,9,4)\in{\mathfrak S}_{11}^\ast$. 
Thus $\cA$ is realized uniquely up to the $\on{GL}(3,\bC)$-action. 
\end{subsection} 
\begin{subsection}{Verifying $\cA\in\cS_{11}\subset\cR_{11}$}\label{S11V} 
We check the freeness of $\cA$ realized in \S \ref{S11R} and 
show that $\cA\in\cR_{11}$. 
We set $\cA_1=\cA\cup\{H_{12}\}$ where 
$h_{12}=x-y$. Then we have 
$$\cA_1\cap H_{12}=\left\{ 
(0, 0),\ 
(1, 1),\ 
\left(\frac12, \frac12\right),\ 
(\zeta+1, \zeta+1),\ 
\left(\frac{\zeta+1}{2},\frac{\zeta+1}{2}\right),\ 
H_{12}\cap H_{\infty} 
\right\}.$$ 
Since $\mu_{\cA_1}=\mu_{\cA}+6$, we have 
$\chi({\cA_1},t)=(t-1)(t-5)(t-6)$. 
Thus $\cA_1\in\cF_{12}$ with $\exp(\cA_1)=(1,5,6)$ 
by Theorem \ref{ABT}, 
and hence $\cA\in\cS_{11}$. 
We set $\cA_2=\cA_1\setminus\{H_{\infty}\}$ and 
$\cA_3=\cA_2\setminus\{H_2\}$. 
Since $n_{\cA_1,H_{\infty}}=6$, we have 
$\cA_2\in\cF_{11}$ with $\exp(\cA_2)=(1,5,5)$. 
Since $n_{\cA_2,H_2}=6$, we have 
$\cA_3\in\cF_{10}=\cI_{10}$. 
Therefore, 
$\cA_2\in\cI_{11}$, $\cA_1\in\cI_{12}$ 
and 
$\cA\in\cR_{11}$. 
\begin{rem} 
Note that $\cA\in\cF_{11}$ satisfying $F(\cA)=[10,5,5]$ 
does not necessary belong to $\cS_{11}$.  For example, 
the arrangement $\cA$ defined by the equation 
$xyz(x^2-z^2)(y^2-z^2)$ 
$(x^2-y^2)(x-y+z)(x-y+2z)$ 
satisfies $F(\cA)=[10,5,5]$ 
but $\cA\in\cI_{11}$. 
\end{rem} 
\begin{defn}\label{Pen} 
The arrangement of Example 4.59 in \cite{OT} 
is the cone of the line arrangement consisted of 
5 sides and 5 diagonals of a regular pentagon, 
defined by the equation 
\begin{eqnarray*} 
\varphi_{\on{pen}}&=& 
z\ (4x^2+2x-z) 
\\&& 
(x^4-10x^2y^2+5y^4+6x^3z-10xy^2z+11x^2z^2-5y^2z^2+6xz^3+z^4) 
\\&& 
(x^4-10x^2y^2+5y^4-4x^3z+20xy^2z+6x^2z^2-10y^2z^2-4xz^3+z^4). 
\end{eqnarray*} 
An arrangement in $\bC^3$ 
is called {\it pentagonal} 
if it is $\on{GL}(3,\bC)$-equivalent to 
$(\varphi_{\on{pen}}=0)$. 
\end{defn} 
It is easy to see that 
$\cA=(\varphi_{\on{pen}}=0)$ satisfies 
$\cA\in\cS_{11}$ and $F(\cA)=[10,5,5]$.  Therefore, 
$$\cS_{11}=\{\text{pentagonal arrangements}\}\subset\cR_{11}.$$ 
By the description in \S \ref{S11R}, the lattice of 
a pentagonal arrangement is realized over $\bQ[\sqrt{5}]$. 
\end{subsection} 
\begin{subsection}{Addition to $\cA\in\cS_{11}$} 
The structures of $\cF_{11}$ and $\cF_{12}$ are given as below. 
\begin{prop}\label{11-12} 
\item{(1)} 
Let $H\in \cA_1\in\cF_{12}$ such that 
$\cA=\cA_1\setminus\{H\}\in\cS_{11}$. 
Then, $\cA_1\in\cI_{12}$ and $\cA_1$ 
has two possibilities 
up to the $\on{GL}(3,\bC)$-action. 
\item{(2)} 
$\cF_{11}=\cR_{11}=\cI_{11}\sqcup\cS_{11}$ and 
$\cF_{12}=\cI_{12}\sqcup\cS_{12}$. 
\end{prop} 
\begin{proof} 
\item{(1)} 
We may assume that $\cA$ has the description 
as in \S \ref{S11L} and \S \ref{S11R}. 
By Theorem \ref{ad}, we have $n_{\cA_1,H}=6$. 
Note that 
$F_{H,3}(\cA_1)\leq1$ since $M_2(\cA)\subset H_{11}$ 
and $F_{H,4}(\cA_1)\leq1$ since 
$\mu_{\cA_1,H}=11<1\cdot4+4\cdot2$. 
Thus 
$F_{H}(\cA_1)=[1,5,0,0],[2,3,1,0],[3,2,0,1]$ or $[4,0,1,1]$. 
 
Suppose $F_{H}(\cA_1)=[1,5,0,0]$. 
By the description in \S \ref{S11L}, we have 
$$M_1(\cA)=\left\{ 
\{1,8\},\{4,8\},\{4,5\},\{5,10\},\{1,10\} 
\right\}\cup\left\{ 
\{2,6\},\{2,7\},\{3,7\},\{3,9\},\{6,9\} 
\right\}. 
$$ 
However, since one of the above two sets contains 
$3$ elements of $H\cap M_1(\cA)$, $H$ coincides with 
some $H_i\in\cA$, which is a contradiction. 
Therefore $F_{H}(\cA_1)\neq[1,5,0,0]$. 
 
Suppose $F_{H}(\cA_1)=[2,3,1,0]$. 
Note that the permutation 
$\rho=(1,5,8,10,4)(2,6,9,3,7)$ 
is an element of ${\mathfrak S}_{11}^\ast$, 
and the group $\langle\rho\rangle$ acts on 
$M_2(\cA)$ or $M_3(\cA)$ 
transitively. 
Since $\rho$ is realized by 
the action of $\on{GL}(3,\bC)$ and 
$\on{Gal}(\bQ[\sqrt{5}]/\bQ)$, 
we may assume $\{1,9,11\}\in H$, i.e., $H\parallel(x=0)$. 
On the other hand, by the direct calculation, we have 
$$ 
M_1(\cA)= 
\left\{\begin{array}{lll} 
\{1,8\}=(0, -\zeta), & \{1,10\}=(0, \zeta+1), 
& \{2,6\}=(\zeta-1, 2-\zeta), 
\\ 
\{4,5\}=(-\zeta, 0), 
& \{5,10\}=(1+\zeta,0), 
& \{2,7\}=(2-\zeta,\zeta-1), 
\\ 
\{3,7\}=(\zeta-1, 1), 
& \{3,9\}=(1, 1), 
& \{4,8\}=(\zeta+1, \zeta+1). 
\\ 
\{6,9\}=(1, \zeta-1), 
\end{array}\right\}. 
$$ 
It is easy to see that no three points of $M_1(\cA)$ 
share the same $x$-coordinate, which contradicts to 
$F_{H,2}(\cA_1)=3$. 
Therefore $F_{H}(\cA_1)\neq[2,3,1,0]$. 
 
We show that, for each 
$\Gamma\in\{[3,2,0,1],[4,0,1,1]\}$ 
and for each $P\in M_3(\cA)$, exists 
the unique line $H\not\in\cA$ such that 
$P\in H$ and $F_H(\cA_1)=\Gamma$. 
First we assume $\{1,2,3,4\}\in H$. 
 
Suppose $\Gamma=[3,2,0,1]$. 
Since $H\not\in\cA$, 
we have $H\cap M_1(\cA)=\{\{5,10\},\{6,9\}\}$, 
and hence $H=(x+(1+\zeta)(y-1)=0)$. 
It is easy to check that $F_{H}(\cA_1)=[3,2,0,1]$. 
 
Suppose $\Gamma=[4,0,1,1]$. 
Since $H\not\in\cA$, 
we have $H\cap M_2(\cA)=\{\{7,8,11\}\}$, 
and hence $H=(\zeta x-y+1=0)$. 
It is easy to check that $F_{H}(\cA_1)=[4,0,1,1]$, 
 
We have seen that the unique $H$ exists for each $\Gamma$ 
if $P=\{1,2,3,4\}$.  To have the unique $H$ passing through 
another $P\in M_3(\cA)$, we have only to apply $\rho$ repeatedly. 
Therefore, $\cA_1$'s sharing the same $F_H(\cA_1)$ are transferred 
by the $\on{GL}(3,\bC)$-action. 
 
Next we show $\cA_1\in\cI_{12}$. 
Note that $\cA_1$ in \S \ref{S11V} 
satisfies $F_{H_{12}}(\cA_1)=[3,2,0,1]$ 
and $\cA_1\in\cI_{12}$. It follows that 
$\cA_1\in\cI_{12}$ when $F_{H}(\cA_1)=[3,2,0,1]$. 
We show that $\cA_1\in\cI_{12}$ when 
$F_{H}(\cA_1)=[4,0,1,1]$. We may assume 
$\cA_1=\cA\cup\{H\}$ with $H=(\zeta x-y+1=0)$. 
By Theorem \ref{ad}, we see that 
$\cA_1\in\cF_{12}$ with $\exp(\cA_1)=(1,5,6)$. 
We set $\cA_2=\cA_1\setminus\{H_{5}\}$ and 
$\cA_3=\cA_2\setminus\{H_9\}$. 
Since $n_{\cA_1,H_{5}}=6$, we have 
$\cA_2\in\cF_{11}$ with $\exp(\cA)=(1,5,5)$. 
Since $n_{\cA_2,H_9}=6$, we have 
$\cA_3\in\cF_{10}=\cI_{10}$. 
Therefore, we have 
$\cA_2\in\cI_{11}$ and $\cA_1\in\cI_{12}$. 
Thus we conclude that $\cA_1\in\cI_{12}$ 
for both cases of $F_{H}(\cA_1)$. 
\item{(2)} 
Since $\cF_{10}=\cI_{10}$ and $\cS_{11}\subset\cR_{11}$, we have 
$\cF_{11}=\cI_{11}\sqcup\cS_{11}=\cR_{11}$ by Lemma \ref{Red}. 
Let $\cA\in\cF_{12}\setminus\cS_{12}$.  Then, there exists 
$H\in\cA$ such that 
$\cA'=\cA\setminus\{H\}\in\cF_{11}=\cI_{11}\sqcup\cS_{11}$. 
If $\cA'\in\cI_{11}$, then $\cA\in\cI_{12}$. 
If $\cA'\in\cS_{11}$, we also have $\cA\in\cI_{12}$ by (1). 
Thus $\cF_{12}=\cI_{12}\sqcup\cS_{12}$. 
\end{proof} 
\end{subsection} 
\end{section} 
\begin{section}{Determination of $\cS_{12}$}\label{S12} 
In this section, we show that $\cS_{12}$ consists of 
monomial arrangements assocated to the group $G(4,4,3)$. 
 
\begin{subsection}{Realization of $\cA\in\cS_{12}$} 
 
Let $\cA\in\cS_{12}$. 
By Proposition \ref{Classify}, $F(\cA)=[0,16,3]$. 
 
We show that $F_H(\cA)=[0,4,1]$ for any $H\in\cA$. 
Note that $n_{\cA,H}=5$ for any $H\in\cA$, 
since $\sum_{H\in\cA}n_{\cA,H}=3\cdot16+4\cdot3=5\cdot12$. 
We have $F_{H,3}(\cA)\leq1$ for $H\in\cA$, since 
$\mu_{\cA,H}=11<2\cdot3+3\cdot2$. 
If 
$M_3(H_0,\cA)=\emptyset$ for some $H_0\in\cA$, 
then 
$11=\mu_{P,H_0}=2F_{H_0,2}(\cA)$, 
a contradiction. Thus, for any $H\in\cA$, we have 
$F_{H,3}(\cA)=1$, and hence $F_H(\cA)=[0,4,1]$. 
 
We determine the realization of $\cA$ in $\bP_{\bC}^2$. 
We may set 
$$M_3(\cA)=\{\{1,2,3,4\},\{5,6,7,8\},\{9,10,11,12\}\}.$$ 
Since $F_{H_9,2}(\cA)=4$, we may assume 
$M_2(H_9,\cA)=\{\{1,5,9\},\{2,6,9\},\{3,7,9\},\{4,8,9\}\}$. 
 
We may set $\overline{\{1,2,3,4\}\{5,6,7,8\}} 
\not\in\cA$ as 
the infinity line $H_{\infty}$ and 
$$ 
h_1=x,\ h_2=x-1,\ h_3=x-p,\ h_4=x-q,\ 
h_5=y,\ h_6=y-1,\ h_7=y-s,\ h_8=y-t. 
$$ 
where 
$ 
|\{0,1,p,q\}|=|\{0,1,s,t\}|=4$. 
By choice of $H_9$, we have $s=p$, $t=q$ and $h_9=x-y$. 
 
We set 
$\alpha_i=\prod_{j=0}^3h_{4i-j}$ for $1\leq i\leq 3$. 
Since $F_{H,2}(\cA)=4$ for any $H\in\cA$, 
we have 
$\on{V}(\alpha_1,\alpha_2)=M_2(\cA)\subset\on{V}(\alpha_3)$. 
Therefore we have $\alpha_3\in\sqrt{(\alpha_1,\alpha_2)} 
=(\alpha_1,\alpha_2)$. 
Since $\deg(\alpha_i)=4$ for $i=1,2,3$, there exists 
$a,b\in\bC\setminus\{0\}$ such that $\alpha_3=a\alpha_1+b\alpha_2$. 
Since $\alpha_3\in(x-y)$, we have $b=-a$.  Therefore 
we may set $a=1, b=-1$. Now we obtain 
$$\alpha_3=\alpha_1-\alpha_2=x(x-1)(x-p)(x-q)-y(y-1)(y-p)(y-q).$$ 
Set $\beta=h_{10}h_{11}h_{12}=\alpha_3/(x-y)$.  Then we have 
$$ 
\beta=x^3+x^2y+xy^2+y^3-(p+q+1)(x^2+xy+y^2) 
+(pq+p+q)(x+y)-pq. 
$$ 
Since 
$\beta$ is a symmetric polynomial in $x$ and $y$, 
we may set 
$$ 
\beta=u^{-1}(x+uy+v)(ux+y+v)(x+y-2w). 
\quad(u,v,w,\in\bC,\ u\neq0). 
$$ 
Then we have $\{9,10,11,12\}=(w,w)$ and hence 
$v=-(1+u)w$. 
Comparing coefficients of $x^2y$, $x^2$, $x$ 
and constant terms, we obtain 
$1 = u+1+u^{-1}$, 
$-p-q-1 = -w(u+4+u^{-1})$, 
$pq+p+q = 3u^{-1}(1+u)^2w^2$ 
and $-pq = -2u^{-1}(1+u)^2w^3$. 
Therefore we have 
$$ 
u^2=-1,\quad 
p+q = 4w-1,\quad 
pq+p+q = 6w^2,\quad 
pq = 4w^3. 
$$ 
and hence $2w-1=0, \pm\sqrt{-1}$. 
Now it is easy to show that 
$$ 
\{1,p,q\}= 
\left\{1, 
\pm\sqrt{-1}, 1 \pm \sqrt{-1} 
\right\}, 
\left\{1, 
\frac{1 + \sqrt{-1}}{2}, 
\frac{1 - \sqrt{-1}}{2} 
\right\}. 
$$ 
These 3 possibilities of $(p,q)$ are identified 
by the actions 
$(x,y)\mapsto(x/p,y/p)$ or $(x,y)\mapsto(x/q,y/q)$, 
which corresponds to the changing of scale so as to set 
$H_2$, $H_3$ or $H_4$ to be $(x=1)$. 
Here we adopt $\{p,q\}=\{\sqrt{-1}, 1+\sqrt{-1}\}$. 
Then we have 
$$\alpha_3=(x-y)(x-\sqrt{-1}y- 1) 
(x+y-1-\sqrt{-1})(x+\sqrt{-1}y-\sqrt{-1}).$$ 
 
Now we have obtained the unique realization of $\cA\in\cS_{12}$ 
up to the $\on{GL}(3,\bC)$-action. 
\end{subsection} 
\begin{subsection}{Verifying $\cA\in\cS_{12}\subset\cR_{12}$} 
We set $\cA_1=\cA\cup\{H_{\infty}\}$. 
It is easy to see that 
$n_{\cA_1,H_{\infty}}=6$. 
Since $\mu_{\cA_1}=\mu_{\cA}+6$, we have 
$\chi({\cA_1},t)=(t-1)(t-5)(t-7)$. 
Thus $\cA_1\in\cF_{13}$ with $\exp(\cA_1)=(1,5,7)$ 
by Theorem \ref{ABT}, and hence $\cA\in\cS_{12}$. 
We set $\cA_2=\cA_1\setminus\{H_{9}\}$ and 
$\cA_3=\cA_2\setminus\{H_{10}\}$. 
Since $n_{\cA_1,H_{9}}=6$, we have 
$\cA_2\in\cF_{12}$ with $\exp(\cA_2)=(1,5,6)$. 
Since $n_{\cA_2,H_{10}}=6$, we have 
$\cA_3\in\cF_{11}=\cR_{11}$. 
Therefore, $\cA_2\in\cR_{12}$, $\cA_1\in\cR_{13}$ 
and $\cA\in\cR_{12}$. 
 
\medskip 
 
In fact, to check whether $\cA\in\cF_{12}$ belongs to $\cS_{12}$ or not, 
we have only to check $F(\cA)$. 
\begin{lem} 
If $\cA\in\cF_{12}$ satisfies $F(\cA)=[0,16,3]$, 
then $\cA\in\cS_{12}$. 
\end{lem} 
\begin{proof} 
If $\cA\not\in\cS_{12}$, there exists $H_1\in\cA$ 
such that $n_{\cA,H_1}\geq6$.  Since 
$\sum_{H\in\cA}n_{\cA,H}=5\cdot12$, 
there exists $H_2\in\cA$ 
such that $n_{\cA,H_2}\leq4$.  Since $F(\cA)=[0,16,3]$, 
we have $F_{H_2}(\cA)=[0,1,3]$. Take $H_3\in\cA\setminus\{H_2\}$ 
such that $\mu_{H_2\cap H_3}(\cA)=2$.  Then, since 
$H_3\cap M_3(\cA)=\emptyset$, we have 
$12-1=\mu_{\cA,H_3}=2n_{\cA,H_3}$, a contradiction. 
\end{proof} 
\begin{defn}\label{443} 
An arrangement in $\bC^3$ is called 
a monomial arrangement associated to 
the group $G(4,4,3)$ (see B.1 of \cite{OT}), 
if it is $\on{GL}(3,\bC)$-equivalent to 
$(\varphi_{4,4,3}=0)$, where 
$$\varphi_{4,4,3}= 
(x^4-y^4)(y^4-z^4)(z^4-x^4). 
$$ 
 
\end{defn} 
It is easy to see that 
$\cA=(\varphi_{4,4,3}=0)$ 
satisfies $F(\cA)=[0,16,3]$. 
Therefore, 
$$\cS_{12}=\{ 
\text{monomial arrangements 
associated to the group\ }G(4,4,3)\}\subset\cR_{12}.$$ 
By Proposition \ref{11-12}, we have 
$\cF_{12}=\cS_{12}\sqcup\cI_{12}=\cR_{12}$. 
Thus Theorem \ref{Main} holds for $|\cA|=12$. 
\end{subsection} 
\end{section} 
\begin{section}{Example in $\cF_{13}\setminus\cR_{13}$}\label{13} 
In this section, we construct the example 
$\cA\in\cF_{13}\setminus\cR_{13}$. 
\begin{subsection}{Defining equation of $\cA$} 
Let $\cA_0=\{H_1,\ldots,H_{12}\}$ be the line arrangement 
in $\bC^2$, where $H_i$ is defined by $h_i$ below 
for each $1\leq i\leq 12$, 
with a generic parameter $\lambda\in\bC$.

\begin{minipage}{0.45\hsize} 
\begin{eqnarray*} 
h_1&=&-\sqrt{3}x-y+{\lambda}+1, 
\\ 
h_2&=&2y+{\lambda}+1, 
\\ 
h_3&=&\sqrt{3}x-y+{\lambda}+1, 
\\ 
h_4&=&\sqrt{3}x-y+{\lambda}-2, 
\\ 
h_5&=&-\sqrt{3}x-y+{\lambda}-2, 
\\ 
h_6&=&2y+{\lambda}-2, 
\\ 
h_7&=&2y-2{\lambda}+1, 
\\ 
h_8&=&\sqrt{3}x-y-2{\lambda}+1, 
\\ 
h_9&=&-\sqrt{3}x-y-2{\lambda}+1, 
\\ 
h_{10}&=&({\lambda}+1)y+\sqrt{3}(1-{\lambda})x-{\lambda}^2+{\lambda}-1, 
\\ 
h_{11}&=&\sqrt{3}{\lambda}x+({\lambda}-2)y-{\lambda}^2+{\lambda}-1, 
\\ 
h_{12}&=&(1-2{\lambda})y-\sqrt{3}x-{\lambda}^2+{\lambda}-1. 
\end{eqnarray*} 
\end{minipage} 
\begin{minipage}{0.47\hsize} 
\begin{center} 
\begin{picture}(180,180)(0,0) 
\setlength\unitlength{0.5pt} 
 \put(10,210){\line(1,0){340}} 
 \put(-20,204){$H_6$} 
 \put(10,187){\line(1,0){340}} 
 \put(-20,179){$H_7$} 
 \put(10,147){\line(1,0){340}} 
 \put(-20,140){$H_2$} 
 \put(65,315){\line(3,-5){170}} 
 \put(230,10){$H_5$} 
 \put(93,315){\line(3,-5){160}} 
 \put(260,40){$H_9$} 
 \put(129,335){\line(3,-5){150}} 
 \put(285,75){$H_1$} 
 \put(124,30){\line(3,5){150}} 
 \put(274,290){$H_4$} 
 \put(96,30){\line(3,5){140}} 
 \put(230,275){$H_8$} 
 \put(66,60){\line(3,5){150}} 
 \put(212,320){$H_3$} 
\qbezier(62,67)(192,167)(322,267) 
 \put(325,275){$H_{11}$} 
\qbezier(94,233)(218,187)(342,141) 
 \put(59,235){$H_{10}$} 
\qbezier(136,336)(156,210)(176,84) 
 \put(165,60){$H_{12}$} 
 \end{picture} 
{{\sc Figure 1.}   $\cA_0$ with $\lambda=2/3$} 
\end{center} 
\end{minipage} 
 
\bigskip 
 
Note that $H_{3i-2}\mapsto H_{3i-1}\mapsto H_{3i}$ 
is obtained by the rotation with angle $-2\pi/3$. 
We will show that the cone $\cA=\on{c}\cA_0$ of $\cA_0$ 
satisfies $\cA\in\cF_{13}\setminus\cR_{13}$. 
Set $\cA=\cA_0\cup\{H_{13}\}$ where $H_{13}$ 
is the infinity line $H_{\infty}$. 
By calculation 
(or by reading off from the figure), we see that 
\begin{eqnarray*} 
M_2(\cA)&=&\left\{\{1,6,8\},\{2,4,9\},\{3,5,7\}\right\}, 
\\ 
M_3(\cA)&=&\{ 
\{1,5,9,13\}, 
\{2,6,7,13\}, 
\{3,4,8,13\} 
\}\subset H_{13}, 
\\ 
M_4(\cA)&=&\left\{ 
\{1,4,7,10,11\}, 
\{2,5,8,11,12\}, 
\{3,6,9,10,12\} 
\right\}. 
\end{eqnarray*} 
Other intersection points of $\cA$ form $M_1(\cA)$. 
It follows that 
$F(\cA)=[21,3,3,3]$, $n_{\cA,H}=6$ for any $H\in\cA$, 
and $\chi({\cA},t)=(t-1)(t-6)^2$. 
Thus $\cA\in\cS_{13}$ provided $\cA\in\cF_{13}$. 
\end{subsection} 
\begin{subsection}{Freeness of $\cA$} 
We show that $\cA\in\cF_{13}$ 
in terms of Yoshinaga's criterion in \cite{Y}. 
 
Let 
$(\cA'',m)$ be 
the Ziegler restriction of $\cA$ onto $H_\infty$. 
It is defined by 
$$ 
y^3 (-\sqrt{3}x-y)^3(\sqrt{3}x-y)^3 \{ 
(\lambda+1)y+\sqrt{3}(1-\lambda)x)\} 
\{ 
\sqrt{3}\lambda x+(\lambda-2)y\} 
\{ 
(1-2\lambda)y-\sqrt{3}x\}=0. 
$$ 
By the change of coordinates 
$$ 
u=y+\sqrt{3}x,\ v=y-\sqrt{3}x, 
$$ 
the defining equation of $(\cA'',m)$ becomes 
$$ 
u^3v^3(u+v)^3 (u+\lambda v)(u+v-\lambda u)(\lambda u+\lambda v-v)=0. 
$$ 
Now recall the following. 
 
\begin{prop}[\cite{Y}]\label{yoshinaga} 
Let $\cB$ be a central arrangement in $\bC^3$, $H \in \cB$ and 
let $(\cB'',m)$ be the Ziegler restriction of $\cB$ onto $H$. Assume that 
$\chi(\cB,t)=(t-1)(t-a)(t-b)$, and 
$\exp(\cB'',m)=(d_1,d_2)$. 
Then $\cB$ is free if and only if 
$ab=d_1d_2$. 
\end{prop} 
 
Also, recall that, 
for a central multiarrangement $(\cC,m)$ in $\bC^2$ 
with $\exp(\cC,m)=(e_1,e_2)$ $(e_1 \le e_2)$, 
it holds that $e_1+e_2=|m|=\sum_{H \in \cC} m(H)$ and 
$e_1=\min_{d \in \bZ} \{d \mid D(\cC,m)_d \neq 0\}$. 
This follows from the fact that 
$D(\cC,m)$ is a rank two free module. For example, see \cite{A}. 
 
Now since $\chi(\cA,t)=(t-1)(t-6)^2$, 
it suffices to show that 
every homogeneous derivation of degree five is zero. 
 
Assume that $\theta \in D(\cA'',m)$ is homogeneous of 
degree five and show that $\theta=0$. 
To check it, first, let us introduce a submultiarrangement $(\cB,m')$ 
of $(\cA'',m)$ defined by 
$$ 
u^3v^3 (u+v)^3=0. 
$$ 
The freeness of $(\cB,m')$ is well-known. 
In fact, we can give its explicit basis as follows. 
$$ 
\partial_1=(u+2v)u^3\partial_u- 
(2u+v)v^3\partial_v,\ 
\partial_2=(u+3v)vu^3 \partial_u+ 
(3u+v)uv^3 \partial_v. 
$$ 
So $\exp(\cB,m')=(4,5)$. Since $D(\cB,m') \supset D(\cA'',m)$, 
there are scalars $a,b,c \in \bC$ such that 
$$ 
\theta=(au+bv)\partial_1+c\partial_2. 
$$ 
The scalars $a,b$ and $c$ are determined by the tangency conditions 
to the three lines 
$u+\lambda v=0,\ u+v-\lambda u=0,\ \lambda u+\lambda v-v=0$. 
Then a direct computation shows that $a,b$ and $c$ satisfy 
$$ 
\lambda(\lambda-1) 
\begin{pmatrix} 
\lambda & -1 &-\lambda(\lambda+1)\\ 
1 & \lambda-1 & -(\lambda-1)(\lambda-2)\\ 
\lambda-1 & -\lambda & -\lambda(\lambda-1)(2\lambda-1) 
\end{pmatrix} 
\begin{pmatrix} 
(\lambda^2-\lambda+1)a\\ 
(\lambda^2-\lambda+1)b\\ 
c 
\end{pmatrix}=0. 
$$ 
The above linear equations imply that 
$$ 
-\lambda(\lambda-1) 
(\lambda-2)(2\lambda-1)(\lambda+1)(\lambda^2-\lambda+1) 
\begin{pmatrix} 
(\lambda^2-\lambda+1)a\\ 
(\lambda^2-\lambda+1)b\\ 
c 
\end{pmatrix}=0. 
$$ 
Since $\lambda$ is generic, it holds that $a=b=c=0$. 
Hence 
$\theta=0$, that is to say, $D(\cA'',m)_5=0$. 
Now apply 
Proposition \ref{yoshinaga} to 
show that $\cA$ is free with $\exp(\cA)=(1,6,6)$. 
 
\medskip 
 
We can also construct the basis 
of $D(\cA)$ explicitly and give 
an alternative proof of the freeness of $\cA$ 
by Theorem \ref{saito}. 
However, we omit to describe it here 
because of its lengthy. 
\end{subsection} 
\begin{subsection}{Non-recursive freeness of $\cA$}\label{13NRC} 
We assume that $\cA\in\cR_{13}$ and deduce the contradiction. 
Recall that $\cA\in\cS_{13}$ with $\exp(\cA)=(1,6,6)$. 
Thus, there exists a line 
$H\subset\bP_{\bC}^2$ such that 
$\cA_1=\cA\cup\{H\}\in\cF_{14}$. 
Note that $n_{\cA_1,H}=7$ by Theorem \ref{ad}. 
\begin{paragraph}{\textbf{Step 1}} 
{\it $F_{H,2}(\cA_1)\leq3$.} 
 
Assume that $|H\cap M_1(\cA)|\geq4$. 
By the description of $M_i(\cA)$ for $2\leq i\leq 4$, we have 
$$ 
M_1(\cA)=\left\{\begin{array}{lllllll} 
 \{1,2\}, &\{2,10\}, &\{4,5\}, &\{5,10\}, &\{7,8\}, &\{8,10\}, &\{10,13\}, 
\\ 
 \{2,3\}, &\{3,11\}, &\{5,6\}, &\{6,11\}, &\{8,9\}, &\{9,11\}, &\{11,13\}, 
\\ 
 \{3,1\}, &\{1,12\}, &\{6,4\}, &\{4,12\}, &\{9,7\}, &\{7,12\}, &\{12,13\} 
\end{array}\right\}.$$ 
Note that the points in each column are transferred each other 
by the rule on indices except for $13$ as 
$3i+1\mapsto 3i+2\mapsto 3i+3\mapsto 3i+1$, 
which corresponds to the rotation of lines with angle $-2\pi/3$. 
Note that some row contains 2 points on $H$. 
By symmetry, we may assume 
$$\left|H\cap\left\{ 
 \{1,2\}, \{2,10\}, \{4,5\}, \{5,10\}, \{7,8\}, \{8,10\}, \{10,13\} 
\right\}\right|\geq2.$$ 
Since $H\not\in\cA$, $H$ is one of the following. 
\begin{eqnarray*} 
&& 
\overline{ 
\{1,2\}\{4,5\} 
},\ 
\overline{ 
\{1,2\}\{5,10\} 
},\ 
\overline{ 
\{1,2\}\{7,8\} 
},\ 
\overline{ 
\{1,2\}\{8,10\} 
},\ 
\overline{ 
\{1,2\}\{10,13\} 
},\ 
\overline{ 
\{2,10\}\{4,5\} 
},\ 
\\&& 
\overline{ 
\{2,10\}\{7,8\} 
},\ 
\overline{ 
\{4,5\}\{7,8\} 
},\ 
\overline{ 
\{4,5\}\{8,10\} 
},\ 
\overline{ 
\{4,5\}\{10,13\} 
},\ 
\overline{ 
\{5,10\}\{7,8\} 
},\ 
\overline{ 
\{7,8\}\{10,13\} 
} 
\end{eqnarray*} 
However, it follows by the direct calculation that 
$n_{\cA_1,H}=10$ if $\{10,13\}\in H$ and 
$n_{\cA_1,H}=11$ if $\{10,13\}\not\in H$, 
both contradicting to $n_{\cA_1,H}=7$. 
Thus 
$F_{H,2}(\cA_1)=|H\cap M_1(\cA)|\leq3$. 
\end{paragraph} 
\begin{paragraph}{\textbf{Step 2}} 
{\it $\sum_{i=3}^5F_{H,i}(\cA_1)\leq1$.} 
 
If $\sum_{i=3}^5F_{H,i}(\cA_1)\geq2$, 
then there exist 2 points 
$P,Q\in\bigcup_{i=2}^4M_i(\cA)$ 
such that $H=\overline{PQ}$. 
If $(\mu_P(\cA),\mu_Q(\cA))=(2,3),(2,4),(3,3),(3,4)$ or $(4,4)$, 
we have $\overline{PQ}\in\cA$ by the description of 
$M_i(\cA)$ for $2\leq i\leq4$.  Thus we may assume 
$P,Q\in M_2(\cA)$. By symmetry, we may assume 
$H=\overline{\{1,6,8\}\{2,4,9\}}$. However, 
by calculation, we have $n_{\cA_1,H}=9\neq7$, 
a contradiction. 
\end{paragraph} 
 
\medskip 
 
By the above steps, we see that 
$F_H(\cA_1)=[3,3,0,1]$ or 
$[4,2,0,0,1]$. 
 
\begin{paragraph}{\textbf{Step 3}} 
{\it 
The case of $F_H(\cA_1)=[3,3,0,1]$.} 
 
Since $|H\cap M_3(\cA)|=F_{H,4}(\cA_1)=1$, we may assume 
$\{2,6,7,13\}\in H$ by symmetry. 
Thus $H$ is parallel to $x$-axis. 
Set $B=\bigcup_{i=2,6,7,13}H_i$, where 
$H_2$, $H_6$, $H_7$ are parallel to $x$-axis 
and $H_{13}$ is the infinity line. 
Then we have 
$\left|M_1(\cA)\setminus B\right|=9$. 
We can check by calculation that 
the $y$-coordinates of the points in 
$M_1(\cA)\setminus B$ 
differ each other. 
Thus we have $F_{H,2}(\cA_1)=|H\cap M_1(\cA)|\leq2$, 
a contradiction. 
\end{paragraph} 
\begin{paragraph}{\textbf{Step 4}} 
{\it 
The case of $F_H(\cA_1)=[4,2,0,0,1]$.} 
 
Since $|H\cap M_4(\cA)|=F_{H,5}(\cA_1)=1$, 
we may assume $Q=\{1,4,7,10,11\}\in H$ by symmetry. 
Since $H\not\in\cA$, we have 
$H\cap M_1(\cA)\subset\{\{2,3\}, \{5,6\}, \{8,9\}, \{12,13\}\}$. 
Since $|H\cap M_1(\cA)|=F_{H,2}(\cA_1)=2$, we can take 
$P\in H\cap \{\{2,3\}, \{5,6\}, \{8,9\}\}$. 
Then $H=\overline{PQ}$.  For 
$P=\{2,3\}, \{5,6\}$ or $\{8,9\}$, 
by calculation, we have 
$n_{\cA_1,H}=8\neq7$, 
a contradiction. 
\end{paragraph} 
 
\medskip 
 
Therefore we conclude that $\cA\not\in\cR_{13}$. 
Now the proof of Theorem \ref{Main} is completed. 
 
\bigskip 
 
Below is the code of Maple we used to check the 
calculations in \S \ref{13NRC}. 
\begin{verbatim} 
# Rotate_Function 
RF:=f->simplify(subs( 
    {s=(-x-sqrt(3)*y)/2,t=(sqrt(3)*x-y)/2},subs({x=s,y=t},f) )): 
 
# Seeds for equations by rotation 
S:=[-3^(1/2)*x-y+A+1,3^(1/2)*x-y+A-2,2*y-2*A+1, 
    -3^(1/2)*x*A+y*A+3^(1/2)*x+y-A^2+A-1]; 
 
# All equations 
Eq:=map(f->collect(f,[x,y]), 
    [seq([S[i],RF(S[i]),RF(RF(S[i]))][],i=1..4)] ); 
 
# Intersection Points, Abbreviated version 
IP:=(f,g)->simplify(subs(solve({f,g},{x,y}),[x,y])): 
 P:=(a,b)->IP(Eq[a],Eq[b]): 
 
# Line through 2 points 
H := (A, B)-> simplify( 
     (B[1]-A[1])*y-(B[2]-A[2])*x+(A[1]*B[2]-A[2]*B[1]) ): 
 
# number of intersections of a line T and Eq 
C:=T->nops({simplify(map( 
   f->solve(subs(y=simplify(solve(T,y)),f),x) ,Eq ))[]}): 
 
# number of intersect. of a line through (a,b),(c,d) and Eq 
CC:=(a,b,c,d)->C(H(P(a,b),P(c,d))): 
 
# number of intersect. of a line through (a,b),(10,13) and Eq 
CCC:=(a,b)->C(Eq[10]+ 
       solve(subs({x=P(a,b)[1],y=P(a,b)[2]},Eq[10]+U),U) ): 
 
# Verification for Step 1 
CC(1,2,4,5) ; CC(1, 2,5,10); CC(1, 2,7,8); CC(1,2,8,10); 
CCC(1,2)    ; CC(2,10,4,5) ; CC(2,10,7,8); CC(4,5,7, 8); 
CC(4,5,8,10); CCC(4, 5)    ; CC(5,10,7,8); CCC(7,8)    ; 
 
# Verification for Step 2 
CC(1,6,2,4); 
 
# Verification for Step 3 
S:={map(f->P(f[])[2],[ 
[3, 1], [3,11], [1,12], [4, 5], [5,10], 
[4,12], [8, 9], [8,10], [9,11]          ])[]}: 
nops(S); 
 
# Verification for Step 4 
CC(1,4,2,3); CC(1,4,5,6); CC(1,4,8,9); 
\end{verbatim} 
\end{subsection} 
\end{section}

\end{document}